\documentclass[pdflatex,sn-mathphys-num]{sn-jnl}

\usepackage{graphicx}%
\usepackage{multirow}%
\usepackage{amsmath,amssymb,amsfonts}%
\usepackage{amsthm}%
\usepackage[title]{appendix}%
\usepackage{xcolor}%
\usepackage{textcomp}%
\usepackage{manyfoot}%
\usepackage{booktabs}%
\usepackage{algorithm}%
\usepackage{algorithmicx}%
\usepackage{algpseudocode}%
\usepackage{listings}%

\usepackage{cases}

\theoremstyle{thmstyleone}%
\newtheorem{theorem}{Theorem}
\newtheorem{proposition}[theorem]{Proposition}%

\theoremstyle{thmstyletwo}%
\newtheorem{remark}{Remark}%

\theoremstyle{thmstylethree}%
\newtheorem{definition}{Definition}%

\newtheorem{assumption}{Assumption}
\newtheorem{lemma}{Lemma}

\begin{document}

\title[Article Title]{GLL-type Nonmonotone Descent Methods Revisited under Kurdyka-{\L}ojasiewicz Property}


\author[1]{\fnm{Yitian} \sur{Qian}}\email{yitian.qian@polyu.edu.hk}

\author[2]{\fnm{Ting} \sur{Tao}}\email{taoting@fosu.edu.cn}

\author*[3]{\fnm{Shaohua} \sur{Pan}}\email{shhpan@scut.edu.cn}

\author[1,4]{\fnm{Houduo} \sur{Qi}}\email{houduo.qi@polyu.edu.hk}

\affil[1]{\orgdiv{Department of Data Science and Artificial Intelligence}, \orgname{The Hong Kong Polytechnic University}, \orgaddress{\city{Hong Kong}, \country{China}}}

\affil[2]{\orgdiv{School of Mathematics}, \orgname{Foshan University}, \orgaddress{\city{Foshan}, \country{China}}}

\affil*[3]{\orgdiv{School of Mathematics}, \orgname{South China University of Technology}, \orgaddress{\city{Guangzhou}, \country{China}}}

\affil[4]{\orgdiv{Department of Applied Mathematics}, \orgname{The Hong Kong Polytechnic University}, \orgaddress{\city{Hong Kong}, \country{China}}}


\abstract{The purpose of this paper is to extend the full convergence results of the classic GLL-type (Grippo-Lampariello-Lucidi) nonmonotone methods to nonconvex and nonsmooth optimization.
We propose a novel iterative framework for the minimization of a proper and lower semicontinuous function $\Phi$. The framework consists of the GLL-type 
nonmonotone decrease condition for a sequence, a relative error condition for its augmented sequence with respect to a Kurdyka-{\L}ojasiewicz (KL) function $\Theta$, and a relative gap condition for the partial maximum objective value sequence. The last condition is shown to be a product of the prox-regularity of $\Phi$ on the set of cluster points, and to hold automatically under a mild condition on the objective value sequence. We prove that for any sequence and its bounded augmented sequence together falling within the framework, the sequence itself is convergent.
Furthermore, when $\Theta$ is a KL function of exponent $\theta\in(0, 1)$, the convergence admits a linear rate if $\theta\in(0, 1/2]$ and a sublinear rate if $\theta\in(1/2, 1)$. As applications, we prove, for the first time, that the two existing algorithms, namely the nonmonotone proximal gradient (NPG) method with majorization and NPG with extrapolation both enjoy the full convergence of the iterate sequences for nonconvex and nonsmooth KL composite optimization problems.}

\keywords{nonconvex and nonsmooth optimization, GLL-type nonmonotone descent methods, full convergence, KL property, DC programs.}


\pacs[MSC Classification]{90C26, 65K05, 49M27}

\maketitle

\section{Introduction}\label{sec1.0}

It has been widely acknowledged that nonmonotone line search 
methods with theoretical convergence guarantee was initially studied by Grippo, Lampariello and Lucidi \cite{Grippo86} (referred to as GLL-type)
for differentiable optimization in the context of globalizing Newton's method.
One of the other earliest works is \cite{chamberlain1982watchdog}.
Since then, methods with nonmonotone line search strategies have become an important part in popular textbooks on numerical optimization, see, e.g., \cite{nocedal1999numerical,conn2000trust}.
Such methods have been developed and extensively used in modern applications from machine learning, 
which include those formulated as sparse optimization using $\ell_0$-norm. 
We refer to the paper \cite{Wright09} by Wright et.al. and many papers that cited it for 
a wide range of applications and new methods.
However, convergence theory is lacking for the fast-growing development in GLL-type methods
for nonconvex and nonsmooth optimization.
This issue has recently been tackled by Qian and Pan \cite{QianPan23} essentially for the weakly convex case,
which severely limits the applications of the obtained results to many optimization problems.
The purpose of this paper is to revisit the GLL-type methods and report a set of general convergence
results that significantly improve those in existing literature as we briefly review below.

\subsection{Related work}\label{sec1.1}

Let $\mathbb{X}$ represent a finite dimensional real vector space endowed with the inner product $\langle \cdot,\cdot\rangle$ and its induced norm $\|\cdot\|$.
Given a proper and lower semicontinuous (lsc) function $\Phi\!:\mathbb{X}\to\overline{\mathbb{R}}:=(-\infty,\infty]$ that is bounded from below,
we consider the minimization problem 
\begin{equation}\label{abstract-prob}
 \min_{x\in\mathbb{X}}\; \Phi(x).
\end{equation} 
Suppose $x^k$ is a current iterate. A GLL-type method generates the next iterate $x^{k+1}$ by
satisfying the growth condition:
\begin{equation} \label{GLL-Condition}
\Phi(x^{k+1} ) \le \max_{j \in \{ [k-m]_+, \ldots, k\}} \Phi(x^j) -\sigma (\| x^{k+1} - x^k \|) ,
\end{equation} 
where $m \ge 0$ is given integer, $[\tau]_+ := \max\{0, \tau\}$ for a given number $\tau$ 
(``$:=$'' means ``define''), and $\sigma(\cdot)\!: \mathbb{R}_+ \to \mathbb{R}_+$ is known as
the forcing function such that for a positive sequence $\{t_k\}_{k\in\mathbb{N}}$, if $\sigma(t_k) \rightarrow 0$ we must have
$t_k \rightarrow 0$ \cite{Grippo02}. A typical choice is $\sigma(t) = \gamma t^2/2$ for some $\gamma >0$.
In most cases, $x^{k+1} = x^k + \alpha_k d^k$ with $d^k$ being a search direction and $\alpha_k >0$ being 
a steplength chosen to satisfy \eqref{GLL-Condition}.
This is the reason why we call it a GLL-type line search method. It is important to note that $x^{k+1}$ may take other formats through, say, a trust-region technique \cite{conn2000trust}. Hence, GLL-type \eqref{GLL-Condition} is suitable for many algorithmic frameworks. For example, 
the GLL-type nonmonotone line search strategy has been integrated into gradient-type methods for smooth optimization (see, e.g., \cite{Raydan97,Birgin00,Grippo02}), proximal gradient (PG) methods for composite optimization  \cite{Wright09,LuZhang12,Kanzow22}, and difference-of-convexity (DC) algorithms for DC programs \cite{LiuPong19,LuZhou19}. 

Some convergence results have been established for GLL-type methods. For instance, when $\Phi(x) = f(x) + g(x)$ with $f(\cdot)$ being continuously differentiable and $g(\cdot)$ being convex, the GLL-type algorithm \texttt{SpaRSA} in \cite{Wright09} generates a sequence whose accumulation points must be critical points \cite[Theorem~1]{Wright09}. 
This is not the type of convergence we are addressing in this paper. Instead, we focus on the (full) convergence of the whole iterates and, when this happens, its convergence rate. As we will see below, such results are few and far between in existing literature.

For smooth optimization problems, Grippo et al. \cite{Grippo86,Grippo89} achieved the convergence of the iterate sequence generated by the GLL-type by assuming that the number of stationary points is finite, which is very restricted and implies the isolatedness of stationary points.
Dai \cite{Dai02} later proved that the objective value sequence of any iterative method with the GLL-type is R-linearly convergent if its objective function is strongly convex. Full convergence seems very challenging beyond convexity and/or isolatedness of solutions. Recently, Qian and Pan \cite{QianPan23} proposed an iterative framework, comprising the GLL-type and a relative error condition, for the minimization of a proper lsc KL function, and proved that any bounded iterate sequence obeying this framework is convergent if the condition \cite[Eq.~(2.6)]{QianPan23} is met. Furthermore, this condition was shown to be sufficient and necessary for KL functions that are locally weakly convex on the set of critical points (see \cite[Proposition~2.5]{QianPan23}). 
This condition is automatically satisfied for Lasso  type problems with convex regularizers considered in
\cite{Wright09}. Consequently the \texttt{SpaRSA} algorithm has full convergence,
which is much stronger than the known convergence result for \texttt{SpaRSA}. This also partly explains 
its observed efficiency.
Unfortunately, only KL weakly convex functions with a restricted weakly convex parameter are known to satisfy 
the condition (see \cite[Lemma~2.6]{QianPan23}), and such KL functions 
are relatively rare in nonconvex optimization. 
 
Therefore, if we are to work with KL property, we must break the weak convexity in order to derive full convergence
for a wider class of problems. 
It turns out that the class of prox-regular functions bears what we would need.
The prox-regular function studied by Poliquin and Rockafellar \cite{Poliquin96} 
belongs to a rich class of 
(nonconvex) functions, covering the locally weakly convex function, the lower-$\mathcal{C}^2$ function, the strongly amenable function, and the sparsity-inducing function such as zero-norm, $\ell_p$-norm with $p\in\!(0,1)$, and rank function. The main purpose of this paper is to design an iterative framework of GLL-type methods for 
prox-regular KL optimization problems, with a view that the obtained full convergence results will apply to a wide class of nonconvex and nonsmooth optimization including DC programming.


\subsection{General iterative framework}\label{sec1.2}

In order to describe the forthcoming iterative framework, we need to make available more information the iterates carry.
We rely on two pieces of them. The first is about the information associated with the current iterate $x^k$ such as the last iterate $x^{k-1}$ in PG methods with extrapolation or the subgradient $\xi^{k}$ of some convex function at $x^k$ in DC algorithms. To represent this piece of information, we create a new variable $z =(x, y) \in \mathbb{Z} :=\mathbb{X} \times \mathbb{Y}$ with the first element being the original variable $x$ itself and the second element being $x^{k-1}$, $\xi^k$, or something else depending how the iterates were actually generated. Here $\mathbb{Y}$ is a finite dimensional space (e.g., the dual space of $\mathbb{X}$ 
if $\xi$ is used for $y$). For $z \in \mathbb{Z}$, we use  $\mathcal{P}_{\mathbb{X}}(z)=x$ to represent the projection of $z$ onto $\mathbb{X}$.

The second piece of information we use is about a merit function $\Theta\!:\mathbb{Z} \to \overline{\mathbb{R}}$, which is often an augmented function of $\Phi$ through a proximal term. One requirement is that $\Theta(\cdot)$ should become closer to $\Phi(\cdot)$ as the algorithm progresses. We now quantify those qualitative arguments. 
For convenience of our description, we let
\[
  \ell(k) := \mbox{the largest index in} \ \arg\max \left\{ \Phi(x^j) \ | \ j = [k-m]_+, \ldots, k    \right\} .
\]
This is in contrast to the usual setting that $\ell(k)$ could be any index that achieves the maximum of 
the function value $\Phi(x^j)$ over the last $(m+1)$ iterates, see \cite[Eq.~(32)]{Wright09}.
 
Suppose a GLL-type nonmonotone descent method for \eqref{abstract-prob} generates a sequence $\{x^k\}_{k\in\mathbb{N}}$ and its augmented sequence $\{z^k\}_{k\in\mathbb{N}}\subset\mathbb{Z}$. We assume that  they comply with the following conditions.

 \begin{description}
 \item [{\bf H1.}] $\Phi(x^{k+1})+a\|x^{k+1}\!-x^k\|^2\le\Phi(x^{\ell(k)})$ for some $a>0$ and for all $k \in \mathbb{N}$.

 \item[{\bf H2.}] For any convergent subsequence $\{x^{\ell(k_j)}\}$, 
 $\limsup\limits_{j\to\infty}\Phi(x^{\ell(k_j)})\!\le\!\Phi(\lim\limits_{j\to\infty}x^{\ell(k_j)})$.

 \item[{\bf H3.}] There exists a proper and lsc KL function $\Theta\!:\mathbb{Z}\to\overline{\mathbb{R}}$ such that $\Theta(z)\ge\Phi(\mathcal{P}_{\mathbb{X}}(z))$ for all $z\in Z^*$, where $Z^*$ denotes the set of cluster points of $\{z^k\}_{k\in\mathbb{N}}$, and 
 \begin{subnumcases}{}\label{Phi-Theta}
 \Phi(x^{k+1})\le\Theta(z^{k+1})\le\Phi(x^{\ell(k)})+\varsigma_k\ \ {\rm for}\ \varsigma_k\ge 0,\lim_{k\to\infty}\varsigma_k=0;\!\!\!\!\!\!\!\!\!\\
  \label{dist-Theta}
  {\rm dist}(0,\partial\Theta(z^k))\le b\|x^k\!-\!x^{k-1}\|\ \ {\rm for\ some}\ b>0.
  \end{subnumcases}
 \item[{\bf H4.}] There exist $\tau\in(0,1),\mu>0$ and $\mathbb{N}\ni\overline{k}>m$ such that for each $k\ge\overline{k}$ and all $i\in\mathbb{N}$ with $\ell(k-1)+1\le i\le\ell(k)\!-\!1$,
 \begin{equation*}
 \sqrt{\Phi(x^{\ell(k)})-\Phi(x^i)}\le \tau\sqrt{a}\,\|x^{i}\!-\!x^{i-1}\|+\mu\sum_{j=i+1}^{\ell(k)}\!\!\|x^j\!-\!x^{j-1}\|.
 \end{equation*} 
 \end{description} 
 
 Some brief comments are as follows. 
 (i) Condition H1 is the GLL search condition \eqref{GLL-Condition} with the forcing function being the quadartic function.
 (ii) Condition H2 is rather weak and does not require $\Phi$ itself to be continuous relative to its domain. 
 (iii) 
 Condition H3 requires the existence of a proper and lsc KL function $\Theta$ such that its function value sequence $\{\Theta(z^k)\}_{k\in\mathbb{N}}$ is bounded by the objective value sequence $\{\Phi(x^k)\}_{k\in\mathbb{N}}$, and the relative inexact optimality of every $z^k$ to $\min_{z\in\mathbb{Z}}\Theta(z)$ is measured by $\|x^{k}\!-\!x^{k-1}\|$. If the objective function $\Phi$ satisfies \eqref{dist-Theta}, it becomes a natural choice of $\Theta$. As will be demonstrated in Section \ref{sec4.1}, potential functions are often a candidate 
 for the function $\Theta$. 
 (iv) Condition H4 aims at restricting the gap of all $\Phi(x^i)$ for $i\in\mathbb{N}$ with $\ell(k\!-\!1)\!+1\le i\le\ell(k)\!-\!1$ from their maximum $\Phi(x^{\ell(k)})$. This is a special one for GLL-type nonmonotone descent methods. 
 For the monotone case, i.e., $m=0$, condition H4 automatically holds because there is no $i$ such that
  $\ell(k\!-\!1)\!+1\le i\le\ell(k)\!-\!1$.
 We also note that even for this case, the iterative framework H1-H4 is new and is different from the monotone one in \cite{Attouch13}. When $m>0$, the term $\Phi(x^{\ell(k)})-\Phi(x^i)$ in H4 is nonnegative under H1, since $\Phi(x^{\ell(k)})\ge\Phi(x^{\ell(i+m)})\ge\Phi(x^i)$ follows Lemma \ref{lemma1-Phi} (i) later and $k\le i+m$.
 
 \subsection{Main contributions}
 
 The rest of the paper is to assess the validity of the framework H1-H4. We address three key aspects. Firstly, we prove that any iterate sequence $\{x^k\}_{k\in\mathbb{N}}$ and its associated $\{z^k\}_{k\in\mathbb{N}}$ conforming to H1-H4 has full convergence and a linear rate under certain KL properties. We then address how condition H4 is satisfied by a large class of functions. Finally, we verify that actual algorithms can be designed to meet the requirements in H1-H4. We do so by studying two existing algorithms for nonconvex and nonsmooth composite problems, which have a wide range of applications in a host of fields such as statistics, machine learning, signal processing, and so on; see \cite{Fan01,LiuPong19,Sra12}. Since this paper focuses on the convergence analysis of GLL-type nonmonotone descent methods that have been applied numerically to a wide variety of problems, we do not include numerical tests for them. We summarize the main contributions below. 
 \begin{description}
 	
 \item[{\bf(1)}] For any sequences $\{x^k\}_{k\in\mathbb{N}}$ and $\{z^k\}_{k\in\mathbb{N}}$ conforming to H1-H4, we prove that $\{x^k\}_{k\in\mathbb{N}}$ is convergent as long as the sequence $\{z^k\}_{k\in\mathbb{N}}$ is bounded (Theorem~\ref{KL-converge}).
 Furthermore, if in addition $\Theta$ is a KL function of exponent $\theta\in(0, 1)$, the convergence admits a linear rate for $\theta\in(0, 1/2]$ and a sublinear rate for $\theta\in(1/2, 1)$ (Theorem~\ref{KL-rate}).

 \item[{\bf(2)}] While the condition H1-H3 can be directly controlled through algorithmic design,
 H4 appears a bit too technical and hard to materialize. 
 We prove that it can be satisfied under rather weak conditions.
 For example, if the difference of the consecutive  objective values can be upper bounded by $\|x^{k+1}\!-\!x^k\|^2$, condition H4 is proved to automatically hold (Lemma~\ref{sequence-cond}). 
 If $\Theta$ is allowed to be $\Phi$, H4 is proved to be a consequence of the prox-regularity of $\Phi$ on the set of cluster points of the sequence $\{x^k\}_{k\in\mathbb{N}}$ (Propositions~\ref{prox-regular} and \ref{prop-ass0}). 
 In view of the universality of prox-regular functions, the proposed iterative framework H1-H4 greatly 
 generalizes the one studied in \cite{QianPan23}, which consists of the condition H1 and H3(\ref{dist-Theta}). 
Due to condition H4, our approach for convergence analysis technique is completely different from that of \cite{QianPan23}; see Remark \ref{Remark-GLL}.

 \item[{\bf(3)}] 
To resolve the concern that conditions H1-H4 may potentially conflict with each other,
we show that two existing algorithms satisfy our conditions and hence establish their full convergence.
This is much stronger than the reported convergence results.
The first algorithm is the NPG$_{\rm major}$  studied in \cite{LiuPong19} for the nonconvex and nonsmooth composite problem \eqref{DC-prob} and the second is 
the PGnls studied in \cite{Kanzow22} for the composite problem \eqref{DC-prob} with $h\equiv 0$.
Moreover, the information needed for our convergence results is much weaker than that required by the
respective algorithms. 
For example, our PGenls neither requires the gradient of $f$ in \eqref{DC-prob} to be Lipschitz continuous nor involves its Lipschitz constant. Hence, it is applicable to more scenarios.


\end{description}
 
\subsection{Organization}
 
The paper is organized as follows. 
In next section, we collect some technical results with a detailed description of the notation used. In Section~\ref{sec3}, we study consequences of the conditions H1-H4 and eventually prove the full convergence and analyze its convergence rate of the generated sequence.
Section~\ref{sec4} achieves the full convergence of the iterate sequences produced by two existing algorithms: ${\rm NPG}_{\rm major}$ for the nonconvex and nonsmooth composite problem \eqref{DC-prob} in \cite{LiuPong19} and PGnls for \eqref{DC-prob} with $h\equiv 0$ in \cite{Kanzow22}. This result is new and enhances the existing convergence results for the two methods. Section~\ref{sec5} concludes the paper.


\section{Preliminaries} \label{Section-Preliminary} 

In this part, we first describe the notation used, followed by the definitions of KL property and prox-regular functions. 
We then present some technical lemmas for future use.

\subsection{Notation} 

The quantities and notation for describing H1 to H4 will be reserved for use throughout the paper and they are: $a,m,b,\tau,\mu,\overline{k},\mathcal{P}_{\mathbb{X}}$ and $Z^*$. We also use other (global) notations. 
For $x \in \mathbb{X}$ and $\epsilon>0$, $\mathbb{B}(x, \epsilon)$ is the $\epsilon$-ball centered at $x$.
For two nonnegative integers $k_1,k_2$, if $k_1\le k_2$, the notation $[k_1,k_2]$ denotes the nonnegative integer set $\{k_1,\ldots,k_2\}$, otherwise it is an empty set. For a real number $t$, the floor operator $\lfloor t \rfloor$ is the largest integer not greater than $t$, and $[t]_{+}$ means $\max\{0,t\}$. For a proper $h\!:\mathbb{X}\!\to\overline{\mathbb{R}}$, $\partial h(\overline{x})$ denotes its (limiting) subdifferential at $\overline{x}\in\mathbb{X}$, and $[\eta_1<h<\eta_2]$ for any $-\infty<\!\eta_1<\!\eta_2<\!\infty$ denotes the set $\{x\in\mathbb{X}\,|\,\eta_1<h(x)<\eta_2\}$. 

\subsection{KL-property and Prox-regular functions}

 First of all, we recall the definition of the KL property (with exponent $\theta\in[0,1)$) for an extended real-valued function $h\!:\mathbb{X}\to\overline{\mathbb{R}}$.
 
 \begin{definition} (\cite[Definition 3.1]{Attouch10}) \label{KL-def}
  For any given $\eta\in(0,\infty ]$, denote by $\Upsilon_{\!\eta}$ the family of continuous concave $\varphi\!: [0, \eta )\rightarrow\mathbb{R}_{+}$ that is continuously differentiable on $(0,\eta)$ with $\varphi'(s)>0$ for all $s\in(0,\eta)$ and $\varphi(0)=0$. A proper function $h\!:\mathbb{X}\to\overline{\mathbb{R}}$ is said to have the KL property at $\overline{x}\in{\rm dom}\,\partial h$ if there exist $\eta\in(0,\infty]$, a neighborhood $\mathcal{U}$ of $\overline{x}$, and a function $\varphi\in\Upsilon_{\!\eta}$ such that for all $x\in\mathcal{U}\cap\big[h(\overline{x})<h<h(\overline{x})+\eta\big]$,
 \[
  \varphi'(h(x)\!-\!h(\overline{x})){\rm dist}(0,\partial h(x))\ge 1.
 \]
 If $\varphi$ can be chosen to be $\varphi(t)=ct^{1-\theta}$ with $\theta\in[0,1)$
  for some $c>0$, then $h$ is said to have the KL property of exponent $\theta$
  at $\overline{x}$. If $h$ has the KL property (of exponent $\theta$) at every point of ${\rm dom}\,\partial h$, then it is called a KL function (of exponent $\theta$).
 \end{definition}

 According to \cite[Lemma 2.1]{Attouch10}, to prove that a proper and lsc function has the KL property (of exponent $\theta\in[0,1)$), it suffices to check if the property holds at its critical points. As discussed in \cite[Section 4]{Attouch10}, the KL property is ubiquitous and the functions definable in an o-minimal structure over the real field admit this property.

 The prox-regularity of an extended real-valued function will be used later. Here we recall its formal definition, and more discussions are seen in \cite[Chapter 13]{RW98}. 
\begin{definition}\label{prox-reg-def}
 (\cite[Definition 13.27]{RW98}) A function $h:\mathbb{X}\to\overline{\mathbb{R}}$ is prox-regular at $\overline{x}$ for $\overline{v}$ if $h$ is finite and locally lsc at $\overline{x}$ with $\overline{v}\in\partial h(\overline{x})$, and there exist $\varepsilon>0$ and $r\ge0$ such that, whenever $(x,v)\in\mathbb{B}((\overline{x},\overline{v}),\varepsilon)\cap{\rm gph}\,\partial h$ with $h(x)<h(\overline{x})+\varepsilon$, 
 \begin{equation*}
 h(x')\ge h(x)+\langle v,x'-x\rangle-({r}/{2})\|x'-x\|^2\quad\ \forall x'\in\mathbb{B}(\overline{x},\varepsilon).
 \end{equation*}
 When this holds for all $\overline{v}\in\partial h(\overline{x})$, $h$ is said to be prox-regular at $\overline{x}$.
 \end{definition}
  
\subsection{Technical lemmas}

 Next we provide three lemmas that are used for the subsequent analysis. Among others, the results of Lemma \ref{lemma1-Phi} are implied by the proof of \cite[Lemma 2.2]{QianPan23}, and the result of Lemma \ref{lemma2-Phi} is just for the proof of Proposition \ref{prop1-Phi}. 
\begin{lemma}\label{lemma2-sequence}
 (see \cite[Lemma 2.7]{QianPan23}) Let $\{\beta_{l}\}_{l\in\mathbb{N}}\subset\mathbb{R}_{+}$ be a nonincreasing sequence such that for all $l>\overline{l}$ with an $\overline{l}\in\mathbb{N}$, $\beta_{l}\le C\max\big\{l^{\frac{1-\theta}{1-2\theta}},(\beta_{l-m-1}\!-\!\beta_{l})^{\frac{1-\theta}{\theta}}\big\}$, where $\theta\in(1/2,1)$ and $C>0$ are the constants. Then there exists $\gamma_0>0$ such that for all $l>\overline{l}$, 
$\beta_{l}\le\max\big\{C,\gamma_0^{\frac{1-\theta}{1-2\theta}}\big\}\lfloor\frac{l-\overline{l}}{m+2}\rfloor^{\frac{1-\theta}{1-2\theta}}$. 
\end{lemma}
 \begin{lemma}\label{lemma1-Phi}
 Let $\{x^k\}_{k\in\mathbb{N}}$ be a sequence satisfying H1. The following results hold. 
 \begin{description}
 \item [(i)] $\{\Phi(x^{\ell(k)})\}_{k\in\mathbb{N}}$ is nonincreasing, and convergent with limit denoted by $\varpi^*$. 

 \item [(ii)] $\lim_{k\to\infty}\|x^{\ell(k)}-x^{\ell(k)-1}\|=0$. 

 \item[(iii)] If there exists ${k}_0\in\mathbb{N}$ such that $\Phi(x^{\ell(k)})=\Phi(x^{\ell({k}_0)})$ for all $k\ge {k}_0$, then all $x^{k}$ with $k\ge {k}_0$ are the same. 
 \end{description}
 \end{lemma}

Having established the convergence of the sequence $\{\Phi(x^{\ell(k)})\}$,
the next lemma says that for $j \ge 0$, under certain conditions, the function value sequence
$\{\Phi(x^{\ell(k)-j}) \}$ shares the limit of $\Phi(x^{\ell(k)})$.

 \begin{lemma}\label{lemma2-Phi}
  Let $\{x^k\}_{k\in\mathbb{N}}$ be a sequence satisfying H1-H2. For any given $j\in\mathbb{N}$, let $\mathcal{K}\subset\mathbb{N}$ be such that $\lim_{\mathcal{K}\ni k\to\infty}\!\Phi(x^{\ell(k)-j})\!=\liminf_{k\to\infty}\Phi(x^{\ell(k)-j})$. If there exists an index set $\mathcal{K}_1\subset \mathcal{K}$ such that $\lim_{\mathcal{K}_1\ni k\to\infty}x^{\ell(k)-j}=\lim_{\mathcal{K}_1\ni k\to\infty}x^{\ell(k)}$, then $\lim_{k\to\infty}\Phi(x^{\ell(k)-j})=\varpi^*$.
 \end{lemma}

 \begin{proof}
 The following inequality from H1 is very useful to the proof. For any nonnegative integers $t,\nu$ such that
 $\nu\le \ell(t)$, replacing $k$ in H1 by $(\ell(t)-\nu)$ results in
 \begin{equation} \label{H1-C}
 \Phi (x^{\ell(t)-\nu+1}) \le \Phi( x^{\ell( \ell(t) - \nu)}) \quad \mbox{for all} \ t, \nu\in\mathbb{N}\ \mbox{with} \
 \nu\le \ell(t).
\end{equation}  
Let $\widetilde{x}=\lim_{\mathcal{K}_1\ni k\to\infty}x^{\ell(k)}$. Using the lower semicontinuity (lsc)  of $\Phi$ and Lemma \ref{lemma1-Phi} (i) yields
 \begin{align*}
  \liminf_{k\to\infty}\Phi(x^{\ell(k)-j})
  &=\lim_{\mathcal{K}_1\ni k\to\infty}\Phi(x^{\ell(k)-j})
  \stackrel{\mbox{(lsc)}}{\ge}
  \Phi(\widetilde{x})\stackrel{\rm H2}{\ge}\limsup_{\mathcal{K}_1\ni k\to\infty}\Phi(x^{\ell(k)})=\varpi^*\\  &=\lim_{k\to\infty}\Phi(x^{\ell(k)})=\lim_{k\to\infty}\Phi(x^{\ell(\ell(k)-j-1)})
  \stackrel{\eqref{H1-C}}{\ge}\limsup_{ k\to\infty}\Phi(x^{\ell(k)-j}),
 \end{align*}
where the last inequality is due to \eqref{H1-C} for $t=k,\nu=j+1$. Hence, $\lim\limits_{ k\to\infty}\Phi(x^{\ell(k)-j})\!=\varpi^*$.  
 \end{proof}

 The following proposition improves the result of \cite[Lemma 2.2 (ii)]{QianPan23} by removing the conditions (2.1)-(2.2) there, as well as those of \cite[Lemma 4]{Wright09}, \cite[Lemma 3.1]{Hager11} and \cite[Proposition 4.1]{Kanzow22} by removing the continuity restriction on the function $\Phi$.
\begin{proposition}\label{prop1-Phi}
 Let $\{x^k\}_{k\in\mathbb{N}}$ be a bounded sequence satisfying conditions H1-H2. Then, it holds $\lim_{k\to\infty}\|x^{k+1}-x^k\|=0,\lim_{k\to\infty}\|x^k -x^{\ell(k)}\|=0$ and $\lim_{k\to\infty}\Phi(x^k)=\varpi^*$. 
 \end{proposition}
 \begin{proof}
 We first prove by induction that for each $j\in\mathbb{N}$ the following limits hold: 
 \begin{equation}\label{aim-limit}
  \lim_{k\to\infty}\|x^{\ell(k)-j+1}-x^{\ell(k)-j}\|=0\ \ {\rm and}\ \ \lim_{ k\to\infty}\Phi(x^{\ell(k)-j})=\varpi^*.
 \end{equation}
 Indeed, when $j=1$, the first limit is due to Lemma \ref{lemma1-Phi} (ii). Let $\mathcal{K}\subset\mathbb{N}$ be an index set such that $\lim_{\mathcal{K}\ni k\to\infty}\!\Phi(x^{\ell(k)-1})\!=\liminf_{k\to\infty}\Phi(x^{\ell(k)-1})$. Note that $\{x^{\ell(k)-\!1}\}_{k\in\mathcal{K}}$ is bounded. There exists $\mathcal{K}_1\!\subset\mathcal{K}$ such that $\{x^{\ell(k)-1}\}_{k\in\mathcal{K}_1}$ is convergent. Along with Lemma \ref{lemma1-Phi} (ii), it holds $\lim_{\mathcal{K}_1\ni k\to\infty}x^{\ell(k)-1}=\lim_{\mathcal{K}_1\ni k\to\infty}x^{\ell(k)}$. Invoking Lemma \ref{lemma2-Phi} for $j=1$ leads to the second limit in \eqref{aim-limit}. 
 Now assume that for some $i\in\mathbb{N}$ the limits in \eqref{aim-limit} hold for all $j\in[1,i]$. We prove that the limits in \eqref{aim-limit} hold for $i+1$. Consider the sequence $\{ x^{\ell(k)-i-1}\}_{k\in\mathbb{N}}$ that is also bounded. We can get a convergent subsequence index ${\mathcal K}$ so that there exists $\widetilde{x}$ such that
\[
 \lim_{\mathcal{K}\ni k\to\infty} x^{\ell(k)-i-1}  = \widetilde{x}
 \quad \mbox{and} \quad
 \lim_{\mathcal{K}\ni k\to\infty}\Phi(x^{\ell(k)-i-1})=\liminf_{k\to\infty}\Phi(x^{\ell(k)-i-1}) .
\]
From condition H1 it follows that 
 \[
  \limsup_{k\to\infty}\|x^{\ell(k)-i}-x^{\ell(k)-i-1}\|^2\le a^{-1}\limsup_{k\to\infty}\big[\Phi(x^{\ell(\ell(k)-i-1)})-\Phi(x^{\ell(k)-i})]=0, 
 \]
 where the equality is by the second limit in \eqref{aim-limit} for $j=i$ and Lemma \ref{lemma1-Phi} (i). Thus, the first limit in \eqref{aim-limit} also holds for $i+1$. Since the first limit in \eqref{aim-limit} holds for all $j\in[1,i+1]$, it holds that
\begin{equation*}
 \lim_{\mathcal{K}\ni k\to\infty}\sum_{j=1}^{i+1}\|x^{\ell(k)-j+1} - x^{\ell(k) - j}\|=0.
\end{equation*}
 Together with $\lim_{\mathcal{K}\ni k\to\infty} x^{\ell(k)-i-1}  = \widetilde{x}$ and the following fact 
\begin{align*}
 x^{\ell(k)}=x^{\ell(k)-i-1}&+\big( x^{\ell(k)-i}\!-\!x^{\ell(k) - i - 1} \big) + \big( x^{\ell(k)-i+1}\!-\!x^{\ell(k)-i} \big)+\cdots+\big( x^{\ell(k)}\!-\! x^{\ell(k)-1} \big),
\end{align*}
we get $\lim_{\mathcal{K}\ni k\to\infty}x^{\ell(k)}=\widetilde{x}$. Invoking Lemma \ref{lemma2-Phi} for $j=i+1$ shows that the second limit in \eqref{aim-limit} holds for $j=i+1$. Then the limits in \eqref{aim-limit} hold for all $j\in\mathbb{N}$.  
 
Equipped with the limits in \eqref{aim-limit}, we reach the final part of our proof.
 For each $\mathbb{N}\ni k\ge m$, noting that $\ell(k)\in[k-m,k]$, we can write $k-m-1=\ell(k)-j$ for some $j\in[1,m+1]$. Using the limit in  \eqref{aim-limit} leads to $\lim_{k\to\infty}\|x^{k+1}-x^k\|=\lim_{k\to\infty}\|x^{k-m}-x^{k-m-1}\|=0$. Note that 
 \[
   x^{\ell(k)}=x^{k-m-1}+\sum_{j=1}^{\ell(k)-(k-m-1)}\!\!\big[x^{\ell(k)-j+1}-x^{\ell(k)-j}\big]\quad{\rm for\ each}\ k\ge m.
 \]
 Along with the first limit in \eqref{aim-limit}, we get $\lim_{k\to\infty}\|x^{\ell(k)}-x^{k-m-1}\|=0$, which along with $\lim_{k\to\infty}\|x^{k+1}-x^k\|=0$ implies $\lim_{k\to\infty}\|x^{\ell(k)}-x^{k}\|=0$. To prove $\lim_{k\to\infty}\Phi(x^k)=\varpi^*$, let $\mathcal{K}\subset\mathbb{N}$ be an index set such that $\lim_{\mathcal{K}\ni k\to\infty}\Phi(x^{k})=\liminf_{k\to\infty}\Phi(x^{k})$. Note that $\{x^{k}\}_{k\in\mathcal{K}}$ is bounded. There exists an index set $\mathcal{K}_1\subset\mathcal{K}$ such that $\{x^{k}\}_{k\in\mathcal{K}_1}$ is convergent with the limit, say, $\widetilde{x}$. Recall that $\lim_{k\to\infty}\|x^{\ell(k)}-x^{k}\|=0$, so $\lim_{\mathcal{K}_1\ni k\to\infty}x^{\ell(k)}=\widetilde{x}$. Thus, by leveraging condition H2, we obtain
 \begin{align*}
  \liminf_{k\to\infty}\Phi(x^{k})&=\lim_{\mathcal{K}_1\ni k\to\infty}\Phi(x^{k})\stackrel{\mbox{(lsc)}}{\ge}\Phi(\widetilde{x})\stackrel{\rm H2}{\ge}\lim_{\mathcal{K}_1\ni k\to\infty}\Phi(x^{\ell(k)})=\varpi^*\\
  &=\limsup\limits_{k\to\infty}\Phi(x^{\ell(k-1)})\stackrel{\rm H1}{\ge}\limsup\limits_{k\to\infty}\Phi(x^{k})\ge\liminf\limits_{k\to\infty}\Phi(x^{k}).
 \end{align*}
 Consequently, $\lim_{k\to\infty}\Phi(x^{k})=\varpi^*$. This finishes our proof.   
%
 \end{proof}

Given that $\Phi(\cdot)$ is assumed to be lsc, condition H2 is equivalent to saying that $\Phi(\cdot)$ is continuous along any convergent subsequence of the particular sequence $\{ x^{\ell(k)}\}$. This is certainly true when $\Phi(\cdot)$ is continuous. When it is not, algorithms need to make sure this continuity along the convergent subsequence of $\{ x^{\ell(k)}\}$. When this happens and the sequence $\{x^k\}_{k\in\mathbb{N}}$ is bounded, the whole function sequence $\{ \Phi(x^k)\}_{k \in \mathbb{N}}$ converges by Proposition \ref{prop1-Phi}. But we like to have the full convergence of the iterate sequence $\{x^k\}_{k \in \mathbb{N}}$ itself. This is the main task of the next section. Although we have proved the convergence of the sequence $\{\Phi(x^k)\}$,
in various places of proofs, we still use $\liminf \Phi(x^k)$ and $\limsup \Phi(x^k)$ on
some subsequence $k \in {\mathcal{K}}$ whenever the lower semicontinuity of $\Phi(\cdot)$ or
H2 is needed.  
 
\section{Full Convergence}\label{sec3}

This section focuses on the full convergence of any sequence $\{x^k\}_{k\in\mathbb{N}}$ if the sequence itself and its augmented sequence $\{z^k\}_{k\in\mathbb{N}}$ conform to H1-H4. To this end,  define
\begin{equation}\label{Xi-Gamk}
 \Xi_k:=\|x^{\ell(k)}-x^{\ell(k)-1}\|\ {\rm and}\  \Gamma_k:=\!\sqrt{\Phi(x^{\ell(k)})-\Phi(x^{\ell(k+1)})}\quad\forall k\in\mathbb{N}.
\end{equation}
Note that $\Gamma_k$ is well defined in view of Lemma~\ref{lemma1-Phi} (i). We will prove the convergence of the 
series $\sum_{k=1}^\infty\Xi_k$ and $\sum_{k=1}^\infty\Gamma_k$ and employ them to achieve the main result.
  
 \subsection{Full convergence of \texorpdfstring{$\{ x^k\}_{k \in \mathbb{N}}$}{xk sequence}} \label{sec3.1}

We already proved the convergence of the function value sequence $\{ \Phi(x^k)\}$ in Proposition~\ref{prop1-Phi}.
We now prove the same for the function value sequence of $\{ \Theta(z^k)\}$ and the convergence of the 
series $\sum_{k=1}^\infty\Xi_k$ and $\sum_{k=1}^\infty\Gamma_k$.
Recall $Z^*$ denotes the set of accumulation points of $\{z^k\}$ and
$\varpi^*$ denotes the limit of the sequence $\{ \Phi(x^k)\}$.

 \begin{proposition}\label{prop2-Phi}
 Let $\{x^k\}_{k\in\mathbb{N}}$ and $\{z^k\}_{k\in\mathbb{N}}$ be the sequences satisfying H1-H3. 
 If $\{z^k\}_{k\in\mathbb{N}}$ is bounded, 
 then the sequence $\{\Theta(z^k)\}_{k\in\mathbb{N}}$ is convergent with the limit $\varpi^*$, and moreover,
 \begin{description}
 \item[(i)] $Z^*$ is nonempty and compact, and $\Theta(z^*)=\varpi^*$ for any $z^*\in Z^*$;

 \item[(ii)] when $\Theta$ is allowed to be $\Phi$, then $Z^*\subset(\partial\Theta)^{-1}(0)=(\partial\Phi)^{-1}(0)$;

 \item[(iii)] $\sum_{k=1}^\infty\Xi_k<\infty$ and $\sum_{k=1}^\infty\Gamma_k<\infty$. 
 \end{description}
 \end{proposition}
 \begin{proof}
 Since $\{x^k\}_{k\in\mathbb{N}}$ is the projection of $\{z^k\}_{k\in\mathbb{N}}$ onto $\mathbb{X}$, the boundedness of $\{z^k\}_{k\in\mathbb{N}}$ implies that of $\{x^k\}_{k\in\mathbb{N}}$, so the conclusion of Proposition~\ref{prop1-Phi} holds. 
 From $\lim_{k\to\infty}\varsigma_k=0$ and the convergence of $\{\Phi(x^k)\}$ by Proposition~\ref{prop1-Phi}, taking limits on both sides of the inequality \eqref{Phi-Theta} leads to 
 $
  \lim_{k \rightarrow \infty} \Theta(z^k) = \varpi^* .
 $ 
We now prove the remaining claims.
 
 {\bf(i)} 
The first part is obvious due to the boundedness of $\{z^k\}$. We prove the second part. Pick any $z^*\in Z^*$. There exists $\mathcal{K}\subset\mathbb{N}$ such that $\lim_{\mathcal{K}\ni k\to\infty}z^k=z^*$. 
It is apparent that $x^*:=\mathcal{P}_{\mathbb{X}}(z^*)=\lim_{\mathcal{K}\ni k\to\infty}x^k$. 
Proposition \ref{prop1-Phi} implies $x^{\ell(k)} \rightarrow x^*$ with $k \in \mathcal{K}$.
The lower semicontinuity of $\Phi$, and condition H2 leads to $\Phi(x^*)=\varpi^*$ because
\[
 \varpi^*=\liminf_{\mathcal{K}\ni k\to\infty}\Phi(x^k)
  \stackrel{\mbox{(lsc)}}{\ge}
  \Phi(x^*)
  \stackrel{\mbox{H2}}{\ge}
  \limsup_{\mathcal{K}\ni k\to\infty}\Phi(x^{\ell(k)})=\varpi^*. 
 \]
On one hand, using the fact $z^*\in Z^*$, the assumption on $\Theta(\cdot)$ in H3 implies
$\Theta(z^*)
\ge \Phi (  \mathcal{P}_{\mathbb{X}} ( z^*) )
= \Phi(x^*)=\varpi^*$. 
On the other hand, the lower semicontinuity of $\Theta(\cdot)$ implies
$\varpi^* = \liminf_{\mathcal{K}\ni k\to\infty}\Theta(z^k) \ge \Theta(z^*)$. Hence, we establish the equality $\Theta(z^*) = \varpi^*$.


 {\bf(ii)} Since $\Theta$ is allowed to be $\Phi$, it suffices to establish the first inclusion. 
 Now pick any $\widetilde{z}\in Z^*$. There exists an index set $\mathcal{K}\subset\mathbb{N}$ such that $\lim_{\mathcal{K}\ni k\to\infty}z^k=\widetilde{z}$. Since $\widetilde{z}\in Z^*$, from item (i), $\Theta(\widetilde{z})=\varpi^*$. Recall that $\lim_{k\to\infty}\Theta(z^k)=\varpi^*$. Then, $\lim_{\mathcal{K}\ni k\to\infty}\Theta(z^k)=\varpi^*=\Theta(\widetilde{z})$.
 From condition H3, for each $k\in\mathcal{K}$, there exists $w^k\in\partial\Theta(z^k)$ such that $\|w^k\|\le b\|x^k-x^{k-1}\|$. Passing the limit $\mathcal{K}\ni k\to\infty$ and using Proposition \ref{prop1-Phi} leads to $\lim_{\mathcal{K}\ni k\to\infty}w^k=0$. Combining $\lim_{\mathcal{K}\ni k\to\infty}\Theta(z^k)=\Theta(\widetilde{z})$ with the outer semicontinuity of $\partial\Theta$ with respect to $\Theta$-attentive convergence (see \cite[Proposition 8.7]{RW98}) then leads to $0\in\partial\Theta(\widetilde{z})$. By the arbitrariness of $\widetilde{z}\in Z^*$, the first inclusion follows.
 
 {\bf(iii)} 
 We first prove $\sum_{k=1}^\infty \Xi_k < \infty$.
 Recall that $\Theta$ is assumed to be a KL function. By invoking \cite[Lemma 6]{Bolte14} with $\Omega= Z^*$ and part (i), there exist $\delta>0,\,\eta>0$ and a function $\varphi\in\Upsilon_{\!\eta}$ such that for all $z\in[\varpi^*<\Theta<\varpi^*+\eta]\cap\big\{z\in\mathbb{Z}\,|\,{\rm dist}(z,Z^*)<\delta\big\}$,
 \[
   \varphi'(\Theta(z)-\varpi^*){\rm dist}(0,\partial\Theta(z))\ge 1.
 \]
 If there exists $\widehat{k}\in\mathbb{N}$ such that 	$\Phi(x^{\ell(\widehat{k})})=\varpi^*$, by Lemma \ref{lemma1-Phi} (i), we have $\Phi(x^{\ell(k)})=\varpi^*$ for all $k\ge\widehat{k}$, and the conclusion follows Lemma \ref{lemma1-Phi} (iii). 
Hence, it suffices to consider that $\Phi(x^{\ell(k)})\ne\varpi^*$ for each $k\in\mathbb{N}$. By \eqref{Phi-Theta} and Lemma \ref{lemma1-Phi} (i),
there exists $\widehat{k}\in\mathbb{N}$ such that 
\begin{align*}
& \Theta(z^{\ell(k)})\ge\Phi(x^{\ell(k)})>\varpi^* \quad \mbox{for all} \ k \in \mathbb{N}, \\
& \Theta(z^{\ell(k)})\le\Phi(x^{\ell(\ell(k)-1)})+\varsigma_{\ell(k)-1}<\varpi^*+\eta \quad \mbox{for all} \ k\ge\widehat{k}.
\end{align*} 
Together with $\lim_{k\to\infty}{\rm dist}(z^{\ell(k)}, Z^*)=0$, for each $k\ge\widehat{k}$ (if necessary by increasing $\widehat{k}$), $z^{\ell(k)}\in[\varpi^*<\Theta<\varpi^*+\eta]\cap\big\{z\in\mathbb{Z}\,|\,{\rm dist}(z,Z^*)<\delta\big\}$, so
 \begin{equation}\label{Phi-ineq0}
  1\le \varphi'(\Theta(z^{\ell(k)})-\varpi^*){\rm dist}(0,\partial\Theta(z^{\ell(k)}))\le b\,\Xi_k\varphi'\big(\Phi(x^{\ell(k)})-\varpi^*\big)
 \end{equation}
 where the second inequality is due to \eqref{dist-Theta} and the nonincreasing of $\varphi'$
 ($\varphi$ is concave). 
Combining \eqref{Phi-ineq0} with Lemma \ref{lemma1-Phi} (i) and the concavity of $\varphi$ on $[0,\eta)$, for all $k\ge\widehat{k}$, we have
 \begin{align*}
 \Phi(x^{\ell(k)})\!-\!\Phi(x^{\ell(k+m+1)})&\le
  b\,\Xi_k\varphi'\big(\Phi(x^{\ell(k)})-\varpi^*\big)[\Phi(x^{\ell(k)})\!-\!\Phi(x^{\ell(k+m+1)})]\\
 &\le b\,\Xi_k\big[\varphi\big(\Phi(x^{\ell(k)})\!-\!\varpi^*\big)
  \!-\!\varphi\big(\Phi(x^{\ell(k+m+1)})\!-\!\varpi^*\big)\big].
 \end{align*}
 On the other hand, since $\ell(k+m+1)\!-1\ge k$, we have $\Phi(x^{\ell(\ell(k+m+1)-1)})\le\Phi(x^{\ell(k)})$ by Lemma \ref{lemma1-Phi} (i), which along with condition H1 implies that 
\begin{equation}\label{philk-descent}
 \Phi(x^{\ell(k+m+1)})+a\|x^{\ell(k+m+1)}\!-\!x^{\ell(k+m+1)-1}\|^2\le\Phi(x^{\ell(k)})\quad\forall k\in\mathbb{N}.
\end{equation}
For every $k\ge\widehat{k}$, adding the above two inequalities together leads to
\begin{align*}
 \|x^{\ell(k+m+1)}\!-\!x^{\ell(k+m+1)-1}\|
&\le \sqrt{ba^{-1}\Xi_k[\varphi(\Phi(x^{\ell(k)})\!-\!\varpi^*)
	\!-\!\varphi(\Phi(x^{\ell(k+m+1)})\!-\!\varpi^*)]}\\
 &\le\frac{1}{4}\Xi_k+ba^{-1}[\varphi(\Phi(x^{\ell(k)})\!-\!\varpi^*)
	\!-\!\varphi(\Phi(x^{\ell(k+m+1)})\!-\!\varpi^*)],
\end{align*}
where the second inequality is due to $\sqrt{\alpha\beta}\le\frac{1}{4}\alpha+\beta$ 
for $\alpha,\beta\ge0$. For any $\nu>k\ge\widehat{k}$, summing the above inequality from $k$ to $\nu$ and using the expression of $\Xi_k$ yields
\begin{align*}
 \sum_{j=k+m+1}^{\nu+m+1}\Xi_j&\le\frac{1}{4}\sum_{j=k}^{\nu}\Xi_j+ba^{-1}\sum_{j=k}^{\nu}[\varphi(\Phi(x^{\ell(j)})\!-\!\varpi^*)
	\!-\!\varphi(\Phi(x^{\ell(j+m+1)})\!-\!\varpi^*)]\\
    &\le\frac{1}{4}\sum_{j=k}^{\nu}\Xi_j+ba^{-1}\sum_{j=k}^{k+m}\varphi(\Phi(x^{\ell(j)})\!-\!\varpi^*),
\end{align*}
where the second inequality is due to the nonnegativity of $\varphi$ on $[0,\eta)$. Note that $\sum_{j=k}^{\nu}\Xi_j\le\sum_{j=k}^{k+m}\Xi_j+\sum_{j=k+m+1}^{\nu+m+1}\Xi_j$. From the above inequality, for any $\nu>k\ge\widehat{k}$,
\begin{equation}\label{rate-ineq1}
 \frac{3}{4}\sum_{j=k+m+1}^{\nu+m+1}\,\Xi_j\le\frac{1}{4}\sum_{j=k}^{k+m}\,\Xi_j+ba^{-1}\sum_{j=k}^{k+m}
 \varphi\big(\Phi(x^{\ell(j)})\!-\!\varpi^*\big).
\end{equation}
 Passing the limit $\nu\to\infty$ to the above inequality results in $\sum_{k=1}^\infty\Xi_k<\infty$. 

Now we prove $\sum_{k=1}^\infty \Gamma_k < \infty$.
 For each $k\ge\widehat{k}$, from the concavity of $\varphi$ on $[0,\eta)$, it immediately follows that   
 \[
  \varphi\big(\Phi(x^{\ell(k)})-\varpi^*\big)
   \!-\!\varphi\big(\Phi(x^{\ell(k+1)})\!-\!\varpi^*\big)\ge \varphi'\big(\Phi(x^{\ell(k)})-\varpi^*\big)[\Phi(x^{\ell(k)})\!-\!\Phi(x^{\ell(k+1)})]
 \]
 which, by invoking the above \eqref{Phi-ineq0} and Lemma \ref{lemma1-Phi} (i), implies 
 \[
   \Phi(x^{\ell(k)})\!-\!\Phi(x^{\ell(k+1)})\le b\,\Xi_k\big[\varphi\big(\Phi(x^{\ell(k)})-\varpi^*\big)
   \!-\!\varphi\big(\Phi(x^{\ell(k+1)})\!-\!\varpi^*\big)\big].
 \]
 Together with the definition of $\Gamma_k$, it follows that for any $k\ge\widehat{k}$, 
 \begin{align*}
 \Gamma_k&\le \sqrt{b\,\Xi_k\big[\varphi\big(\Phi(x^{\ell(k)})-\varpi^*\big)\!-\!\varphi\big(\Phi(x^{\ell(k+1)})\!-\!\varpi^*\big)\big]}\\
  &\le\frac{1}{2}\Xi_k+\frac{b}{2}\big[\varphi\big(\Phi(x^{\ell(k)})-\varpi^*\big)\!-\!\varphi\big(\Phi(x^{\ell(k+1)})\!-\!\varpi^*\big)\big],\nonumber
 \end{align*}
which by the nonnegativity of $\varphi$ on $[0,\eta)$ implies that for any $\nu>k\ge\widehat{k}$,
\begin{equation}\label{KL-Gammak}
 \sum_{j=k}^{\nu}\Gamma_j\le\frac{1}{2}\sum_{j=k}^{\nu}\Xi_j+\frac{b}{2}\varphi\big(\Phi(x^{\ell(k)})-\varpi^*\big).
\end{equation}
 Passing the limit $\nu\to\infty$ and using $\sum_{k=1}^\infty\Xi_k<\infty$ leads to  $\sum_{k=1}^\infty\Gamma_k<\infty$.  
\end{proof}
 
 Next we bound the term $\sum_{j=\ell(k-1)}^{\ell(k)-1}\|x^{j+1}\!-\!x^{j}\|$.
\begin{proposition}\label{prop-H4}
 Let $\{x^k\}_{k\in\mathbb{N}}$ and $\{z^k\}_{k\in\mathbb{N}}$ be the sequences complying with H1 and H4. Then, for each $k\ge\overline{k}$, where $\overline{k}$ is the same as in condition H4, it holds
 \[
  \sum_{j=\ell(k-1)}^{\ell(k)-1}\!\!\|x^{j+1}-x^{j}\| \le c(\mu,\tau,a,m)\bigg[\sum_{j=k-m-1}^{k-1}\!\!\Gamma_j+\Xi_k\bigg],
 \]
 where $c(\mu,\tau,a,m)\!:=\!(m+1)(1\!+\overline{\mu})^{m-1}\max\Big\{\frac{1}{\sqrt{a}(1-\tau)},\frac{1}{(1+\overline{\mu})^{m-1}}+\overline{\mu}\Big\}$ with $\overline{\mu}=\!\frac{\mu}{\sqrt{a}(1-\tau)}$.	
\end{proposition}

\begin{proof}
We will use a few simple facts from the definition of $\ell(k)$. For each $i\ge 1$, we know
$
 \ell(i+m) \in \{ i, i+1, \ldots, i+m\},
$
achieving the largest function value $\Phi(x^j)$ over this index set.
Therefore, we must have $\Phi(x^i) \le \Phi(x^{\ell(i+m)})$ for any $i\ge 1$.
Now let $i\in[\ell(k\!-\!1)+1,\ell(k)\!-\!1]$, we have
$k\ge i-1\ge\ell(k-1)\ge k-m-1$ and hence $i+m \ge k$. By Lemma \ref{lemma1-Phi} (i), we have
\begin{equation} \label{Eq-I}
\Phi(x^{\ell(k)}) \le 	\Phi(x^{\ell(i-1) }) \le \Phi(x^{\ell(k-m-1) }) 
\ \mbox{and }\
\Phi(x^i) \le \Phi(x^{\ell(i+m)}) \le \Phi(x^{\ell(k)}) .
\end{equation} 
	
 Fix any $k\ge\overline{k}$. If $[\ell(k-1)+1,\ell(k)-1]=\emptyset$, the conclusion automatically holds. Hence, it suffices to consider that $[\ell(k-1)+1,\ell(k)-1]\ne\emptyset$. From conditions H1 and H4, for each $i\in[\ell(k-1)+1,\ell(k)-1]$, it holds  
 \begin{align}\label{xi-diff}
 \|x^{i}-x^{i-1}\|&\stackrel{{\rm H1}}{\le}\frac{1}{\sqrt{a}}\sqrt{\Phi(x^{\ell(i-1)})-\Phi(x^{\ell(k)})+\Phi(x^{\ell(k)})-\Phi(x^{i})}\nonumber\\
  & \stackrel{\eqref{Eq-I}}{\le}
  \frac{1}{\sqrt{a}}\left[\sqrt{\Phi(x^{\ell(k-m-1)})-\Phi(x^{\ell(k)})}+\sqrt{\Phi(x^{\ell(k)})-\Phi(x^{i})}\right]\nonumber\\
  &\le\frac{1}{\sqrt{a}}\bigg[\sum_{j=k-m-1}^{k-1}\!\!\Gamma_j+\sqrt{\Phi(x^{\ell(k)})-\Phi(x^{i})}\bigg]\nonumber\\
  &\stackrel{{\rm H4}}{\le} \frac{1}{\sqrt{a}}\!\!\sum_{j=k-m-1}^{k-1}\!\!\Gamma_j+\tau\|x^{i}-x^{i-1}\|+\frac{\mu}{\sqrt{a}}\sum_{j=i+1}^{\ell(k)}\|x^{j}-x^{j-1}\|.
 \end{align}
Recall that $\tau\in(0,1)$. From inequality \eqref{xi-diff}, for each $\nu\in[\ell(k\!-\!1)+1,\ell(k)\!-\!1]$, 
 \begin{equation}\label{term-ineq}
  \|x^{\nu}-x^{\nu-1}\|\le\frac{1}{\sqrt{a}(1\!-\!\tau)}\!\!\sum_{j=k-m-1}^{k-1}\!\!\Gamma_j+\frac{\mu}{\sqrt{a}(1\!-\!\tau)}\sum_{j=\nu+1}^{\ell(k)}\|x^{j}-x^{j-1}\|.
 \end{equation}
 The above \eqref{term-ineq}, along with $\overline{\mu}=\frac{\mu}{\sqrt{a}(1-\tau)}$, shows that for each $l\in[1,\ell(k)-\ell(k\!-\!1)\!-\!1]$, 
 \begin{equation}\label{recursion1}
  \|x^{\ell(k)-l}-x^{\ell(k)-l-1}\|\le\frac{1}{\sqrt{a}(1\!-\!\tau)}\!\!\sum_{j=k-m-1}^{k-1}\!\!\Gamma_j+\overline{\mu}\!\!\sum_{j=\ell(k)-l+1}^{\ell(k)}\!\!\|x^{j}-x^{j-1}\|.
 \end{equation}
 By leveraging the above \eqref{term-ineq} and \eqref{recursion1}, we claim that for all $i\in[1,\ell(k)-\ell(k\!-\!1)\!-\!1]$, 
 \begin{align}\label{aim-ineq1}
 \sum_{j=\ell(k)\!-i+1}^{\ell(k)}\!\!\|x^{j}\!-\!x^{j-1}\|\le \frac{(1+\overline{\mu})^{i-1}\!-\!1}{\overline{\mu}\sqrt{a}(1\!-\!\tau)}\!\!\sum_{j=k-m-1}^{k-1}\!\!\Gamma_j+(1+\overline{\mu})^{i-1}\|x^{\ell(k)}\!-\!x^{\ell(k)-1}\|,\!\!\\
 \label{aim-ineq2} 
 \|x^{\ell(k)-i}-\!x^{\ell(k)-i-1}\|
  \le\frac{(1\!+\!\overline{\mu})^{i-1}}{\sqrt{a}(1\!-\!\tau)}\!\!\sum_{j=k-m-1}^{k-1}\!\!\Gamma_j+(1\!+\!\overline{\mu})^{i-1}\overline{\mu}\|x^{\ell(k)}\!-\!x^{\ell(k)-1}\|.
 \end{align}
 Indeed, when $i=1$, inequality \eqref{aim-ineq1} is trivial, and \eqref{aim-ineq2} follows \eqref{term-ineq} with $\nu=\ell(k)-1$. 
 Now assuming that \eqref{aim-ineq1}-\eqref{aim-ineq2} hold for some $i\in[1,\ell(k)-\ell(k\!-\!1)\!-\!1]$, we prove them for $i+1\in[1,\ell(k)-\ell(k\!-\!1)\!-\!1]$. To this end, using the above \eqref{term-ineq} with $\nu=\ell(k)-i$ leads to
 \begin{align}\label{mid-ineq}
 \sum_{j=\ell(k)\!-i}^{\ell(k)}\|x^{j}-x^{j-1}\|&=\|x^{\ell(k)-i}\!-\!x^{\ell(k)-i-1}\|+\sum_{j=\ell(k)-i+1}^{\ell(k)}\|x^{j}-x^{j-1}\|\nonumber\\
 &\stackrel{\eqref{term-ineq}}{\le}\frac{1}{\sqrt{a}(1\!-\!\tau)}\!\!\sum_{j=k-m-1}^{k-1}\!\!\Gamma_j+(1+\overline{\mu})\!\!\sum_{j=\ell(k)-i+1}^{\ell(k)}\!\!\|x^{j}-x^{j-1}\|\nonumber\\
 &\le\frac{(1+\overline{\mu})^{i}\!-\!1}{\overline{\mu}\sqrt{a}(1\!-\!\tau)}\!\!\sum_{j=k-m-1}^{k-1}\!\!\Gamma_j+(1+\overline{\mu})^{i}\|x^{\ell(k)}-x^{\ell(k)-1}\|
 \end{align}
 where the second inequality is obtained by using \eqref{aim-ineq1}. Now invoking \eqref{recursion1} for $l=i+1$ leads to  
 \begin{align*}
  \|x^{\ell(k)-i-1}-x^{\ell(k)-i-2}\|
  &\le\frac{1}{\sqrt{a}(1\!-\!\tau)}\!\!\sum_{j=k-m-1}^{k-1}\!\!\Gamma_j+\overline{\mu}\!\!\sum_{j=\ell(k)-i}^{\ell(k)}\!\!\|x^{j}-x^{j-1}\|\\
  &\stackrel{\eqref{mid-ineq}}{\le}\frac{(1+\overline{\mu})^{i}}{\sqrt{a}(1\!-\!\tau)}\!\!\sum_{j=k-m-1}^{k-1}\!\!\Gamma_j+\overline{\mu}(1+\overline{\mu})^{i}\|x^{\ell(k)}-x^{\ell(k)-1}\|.
 \end{align*}
 The above two equations show that the claimed \eqref{aim-ineq1}-\eqref{aim-ineq2} hold for $i+1$. The above arguments by induction proves that \eqref{aim-ineq2} holds for all $i\in[1,\ell(k)-\ell(k\!-\!1)\!-\!1]$. Thus, together with the definition of $\Xi_k$ in \eqref{Xi-Gamk}, it follows
 \begin{align*}
 \sum_{j=\ell(k-1)}^{\ell(k)-1}\!\!\!\!\!\|x^{j+1}-x^{j}\|\!
  &=\Xi_k+\sum_{i=1}^{\ell(k)-\ell(k-1)-1}\!\!\|x^{\ell(k)-i}-x^{\ell(k)-i-1}\|\\
  &\le\Xi_k + (m+1){\max_{1\le i\le \ell(k)-\ell(k-1)-1}}\big\|x^{\ell(k)-i}-x^{\ell(k)-i-1}\big\|\\
 &\stackrel{\eqref{aim-ineq2}}{\le}\!\! \frac{(m\!+\!1)(1\!+\!\overline{\mu})^{m-1}}{\sqrt{a}(1\!-\!\tau)}\!\!\!\sum_{j=k-m-1}^{k-1}\!\!\!\!\!\Gamma_j\!+\!\big[1\!+\!(m\!+\!1)(1\!+\!\overline{\mu})^{m-1}\overline{\mu}\big]\Xi_k.
 \end{align*}
 This, by the definition of $c(\mu,\tau,a,m)$, implies the desired result.  
\end{proof}

By using Proposition~\ref{prop2-Phi} and Proposition~\ref{prop-H4}, we are ready to establish the main result. 
\begin{theorem}\label{KL-converge}
 Let $\{x^k\}_{k\in\mathbb{N}}$ and $\{z^k\}_{k\in\mathbb{N}}$ be the sequences satisfying H1-H4. If the sequence $\{z^k\}_{k\in\mathbb{N}}$ is bounded, then $\sum_{k=1}^\infty\|x^{k+1}\!-\!x^k\|<\infty$, and thereby the sequence $\{x^k\}_{k\in\mathbb{N}}$ is convergent with limit, say, $\overline{x}$.
\end{theorem}

\begin{proof}
 For any $\nu>\overline{k}$, where $\overline{k}\in\mathbb{N}$ is the same as in condition H4, it holds 
 \begin{align*}
 \sum_{k=\overline{k}-1}^{\ell(\nu)-1}\|x^{k+1}-x^{k}\|&\le\sum_{k=\ell(\overline{k}-1)}^{\ell(\nu)-1}\|x^{k+1}-x^{k}\|=\sum_{k=\overline{k}}^{\nu}\sum_{j=\ell(k-1)}^{\ell(k)-1}\|x^{j+1}-x^{j}\|\\
 &\le c(\mu,\tau,a,m)\sum_{k=\overline{k}}^{\nu}\Big[\sum_{j=k-m-1}^{k-1}\!\!\Gamma_j+\Xi_k\Big],
 \end{align*}
 where the second inequality is due to Proposition \ref{prop-H4}. 
 It is not hard to check that 
 \[
   \sum_{k=\overline{k}}^{\nu}\sum_{j=k-m-1}^{k-1}\Gamma_j\le(m\!+\!1)\!\!\sum_{j=\overline{k}-m-1}^{\nu-1}\Gamma_j.
 \]
 From the above two inequalities, it is immediate to obtain 
 \[
  \sum_{k=\overline{k}-1}^{\ell(\nu)-1}\|x^{k+1}-x^{k}\|\le(m+1)c(\mu,\tau,a,m)\!\!\sum_{j=\overline{k}-m-1}^{\nu-1}\!\!\Gamma_j+c(\mu,\tau,a,m)\sum_{k=\overline{k}}^{\nu}\Xi_k.
 \]
 Passing $\nu\to\infty$ to this inequality and using Proposition \ref{prop2-Phi} (iii) yields the result.  
 \end{proof}
 
 The boundedness assumption is basic for the full convergence analysis of iterate sequences. The assumption on the boundedness of $\{z^k\}_{k\in\mathbb{N}}$ in Theorem \ref{KL-converge} automatically holds whenever $\mathcal{L}_{\Theta}:=\{z\in\!\mathbb{Z}\,|\,\Theta(z)\le\Phi(x^0)+\sup_{k\in\mathbb{N}}\varsigma_k\}$ is bounded, since according to Lemma \ref{lemma1-Phi} (i) and inequality \eqref{Phi-Theta} it holds $\{z^k\}_{k\in\mathbb{N}}\subset\mathcal{L}_{\Theta}$. In fact, the boundedness of $\{z^k\}_{k\in\mathbb{N}}$ usually comes from that of $\{x^k\}_{k\in\mathbb{N}}$; see Section \ref{sec4}. 
\begin{remark}\label{Remark-GLL}
 It is worth emphasizing that the proof technique of Theorem \ref{KL-converge} is completely different from the one adopted in \cite{QianPan23}, and the partition of the iterate sequence $\{x^k\}_{k\in\mathbb{N}}$ is not used here. 
Technically, \cite{QianPan23} partitioned the iterate sequence $\{x^k\}_{k\in\mathbb{N}}$ into two parts: $\{x^k\}_{k\in\mathcal{K}_1}$ and $\{x^k\}_{k\in\overline{\mathcal{K}}_1}$ with 
 $\mathcal{K}_1\!:=\big\{k\in\mathbb{N}\,|\,\Phi(x^{k+1}) \le \Phi( x^{\ell(k+1)})-\frac{a}{2}\| x^{k+1}-x^k\|^2\big\}$ and $\overline{\mathcal{K}}_1\!:=\mathbb{N}\backslash\mathcal{K}_1$.
Their behaviour is separately analyzed in \cite{QianPan23} to prove the full convergence under the condition (2.6) there. 
As mentioned in the introduction, 
\cite[Condition~(2.6)]{QianPan23} is very restrictive.
To our best knowledge,  only KL weakly convex functions with a restricted weakly convex parameter are known to satisfy this condition. 
 In contrast, the key to the proof of Theorem \ref{KL-converge} is to bound $\sum_{k=1}^{\infty}\|x^{k}-x^{k-1}\|$ with the convergent series $\sum_{k=1}^{\infty}\Xi_k$ and $\sum_{k=1}^{\infty}\Gamma_k$. To achieve this goal, we introduce condition H4 to control the gap of all $\Phi(x^i)$ for $i\in[\ell(k\!-\!1)+1,\ell(k)\!-\!1]$ from $\Phi(x^{\ell(k)})$, and employ H4 to skillfully establish Proposition~\ref{prop-H4} (essentially Theorem~\ref{KL-converge}). 
 As will be shown in Section \ref{sec3.2}, condition H4 is a product of the prox-regularity of $\Phi$ on the set of cluster points of $\{x^k\}_{k\in\mathbb{N}}$ if $\Theta$ is allowed to be $\Phi$, which is applicable to much wider scenarios than the condition (2.6) of \cite{QianPan23}.
 \end{remark}

\subsection{Convergence rate} \label{Subsection-rate}

 To achieve the convergence rate of $\{x^k\}_{k\in\mathbb{N}}$, we need the following result.
\begin{proposition}\label{KL-Xirate}
 Let $\{x^k\}_{k\in\mathbb{N}}$ and $\{z^k\}_{k\in\mathbb{N}}$ be the sequences satisfying H1-H3. Suppose that $\{z^k\}_{k\in\mathbb{N}}$ is bounded, and that $\Theta$ is a KL function of exponent $\theta$. Then, there exist $\widetilde{k}\in\mathbb{N},\widehat{\gamma}>0,\widehat{\varrho}\in(0,1),\widetilde{\gamma}>0$ and $\widetilde{\varrho}\in(0,1)$ such that for all $k\ge\widetilde{k}$,
 \begin{equation}\label{XG_rate}
  \sum_{j=k}^{\infty}\Xi_{j}\le\left\{\begin{array}{cl}
 	\!\widehat{\gamma}\widehat{\varrho}^{k}
       &{\rm if}\ \theta\in(0,\frac{1}{2}],\\
		\widehat{\gamma}{k}^{\frac{1-\theta}{1-2\theta}}&{\rm if}\ \theta\in(\frac{1}{2},1)
	\end{array}\right.\,{\rm and}\ 
	\sum_{j=k}^{\infty}\Gamma_{j}\le\left\{\begin{array}{cl}
		\!\widetilde{\gamma}\widetilde{\varrho}^{k} 
		&{\rm if}\ \theta\in(0,\frac{1}{2}],\\
		\widetilde{\gamma}{k}^{\frac{1-\theta}{1-2\theta}}&{\rm if}\ \theta\in(\frac{1}{2},1).
	\end{array}\right.    
 \end{equation}
\end{proposition}

\begin{proof}
We will only consider $\Phi(x^{\ell(k)})\ne\varpi^*$ for every $k\in\mathbb{N}$. Otherwise, it is trivial by Lemma~\ref{lemma1-Phi} (i) and (iii). From the proof of Proposition~\ref{prop2-Phi}~(iii), the previous \eqref{Phi-ineq0} holds with $\varphi(t)=ct^{1-\theta}$ for all $k\ge\widehat{k}$, where $c>0$ is a constant and $\widehat{k}$ is the same as in the proof of Proposition~\ref{prop2-Phi} (iii). 
Then, for all $k\ge\widehat{k}$,
 \begin{equation}\label{KLtheta-ineq}
  \varphi(\Phi(x^{\ell(k)})\!-\!\varpi^*)
   =c[(\Phi(x^{\ell(k)})\!-\!\varpi^*)^{\theta}]^{\frac{1-\theta}{\theta}}\le c\big[bc(1\!-\!\theta)\big]^{\frac{1-\theta}{\theta}}\Xi_k^{\frac{1-\theta}{\theta}},
 \end{equation}
 where the inequality is obtained by using \eqref{Phi-ineq0}. By substituting the above inequality into the previous \eqref{rate-ineq1}, it follows for any $\nu>k\ge\widehat{k}$, 
 \begin{equation*}
 \frac{3}{4}\sum_{j=k+m+1}^{\nu+m+1}\,\Xi_j\le\frac{1}{4}\sum_{j=k}^{k+m}\,\Xi_j+{bc}a^{-1}\big[bc(1-\theta)\big]^{\frac{1-\theta}{\theta}}	\sum_{j=k}^{k+m}\Xi_j^{\frac{1-\theta}{\theta}}.
\end{equation*}
 Let $\Lambda_k\!:=\sum_{j=k}^{\infty}\Xi_{j}$ for each $k\in\mathbb{N}$. Passing the limit $\nu\to\infty$ to this inequality gives
 \begin{equation}\label{Xik-ineq1}
 \frac{3}{4}\Lambda_{k+m+1}\le\frac{1}{4}\sum_{j=k}^{k+m}\Xi_j+{bc}a^{-1}\big[bc(1-\theta)\big]^{\frac{1-\theta}{\theta}}	\sum_{j=k}^{k+m}\Xi_j^{\frac{1-\theta}{\theta}}.
 \end{equation}
 When $\theta\in(0,1/2]$, since $\frac{1-\theta}{\theta}\ge 1$ and $\Xi_k<1$ for all $k\ge\widehat{k}$ (if necessary by increasing $\widehat{k}$), the following inequality holds with $M=\frac{1}{3}+\frac{4bc}{3a}[bc(1-\theta)]^{\frac{1-\theta}{\theta}}$:
 \[
   \Lambda_{k+m+1}\le M(\Lambda_{k}-\Lambda_{k+m+1}),
 \]
 which implies $\Lambda_{k+m+1}\le\frac{M}{1+M}\Lambda_{k}$ for all $k\ge\widehat{k}$. Consequently, for all $k\ge\widehat{k}+m\!+\!1$,  
 $\Lambda_{k}\le\frac{M}{1+M}\Lambda_{k-m-1}$. 
 According to this recursion formula, we get $\Lambda_{k}\le{\tau_0}^{\lfloor\frac{k-\widehat{k}}{m+1}\rfloor}\Lambda_{\widehat{k}}$ with $\tau_0=\frac{M}{1+M}$, which shows that the first inequality in \eqref{XG_rate} holds for $\theta\in(0,{1}/{2}]$. When $\theta\in(1/2,1)$, since $\Xi_k<1$ for all $k\ge\widehat{k}$ (if necessary by increasing $\widehat{k}$) and $\frac{1-\theta}{\theta}<1$, from the above \eqref{Xik-ineq1}, for any $k\ge\widehat{k}$,  
 \begin{align*}
\Lambda_{k+m+1}\le M\sum_{j=k}^{k+m}\!\Xi_j^{\frac{1-\theta}{\theta}}&\le M(m+1)^{\frac{2\theta-1}{\theta}}\Big[\sum_{j=k}^{k+m}\!\Xi_j\Big]^{\frac{1-\theta}{\theta}}\\
&=M(m+1)^{\frac{2\theta-1}{\theta}}\big(\Lambda_{k}-\!\Lambda_{k+m+1}\big)^{\frac{1-\theta}{\theta}}, 
\end{align*}  
 where the second inequality is due to the concavity of $\mathbb{R}_{++}\ni t\mapsto t^{\frac{1-\theta}{\theta}}$. Hence, for all $k\ge\widehat{k}+m\!+\!1$, $\Lambda_{k}\le M(m\!+\!1)^{\frac{2\theta-1}{\theta}}\big(\Lambda_{k-m-1}\!-\!\Lambda_{k}\big)^{\frac{1-\theta}{\theta}}$. 
 Now invoking Lemma \ref{lemma2-sequence} leads to the first claim in \eqref{XG_rate} for $\theta\in({1}/{2},1)$. 
 
 For the second one, by passing $\nu\to\infty$ to \eqref{KL-Gammak} and using \eqref{KLtheta-ineq}, for any $k\ge\widehat{k}$,
 \begin{align}\label{KLrate-Gammak}
 \sum_{j=k}^{\infty}\Gamma_{j}
 &{\stackrel{\eqref{KL-Gammak}}{\le}}\frac{1}{2}\sum_{j=k}^{\infty}\Xi_j
   +\frac{b}{2}\varphi\big(\Phi(x^{\ell(k)})-\varpi^*\big)
  {\stackrel{\eqref{KLtheta-ineq}}{\le}}\frac{1}{2}\sum_{j=k}^{\infty}\Xi_j
	 +\frac{bc}{2}\big[bc(1-\theta)\big]^{\frac{1-\theta}{\theta}}
	 \Xi_k^{\frac{1-\theta}{\theta}}\nonumber\\
  &{\stackrel{\eqref{philk-descent}}{\le}}\frac{1}{2}\sum_{j=k}^{\infty}\Xi_j
	 +\frac{bc}{2}\big[bc(1-\theta)\big]^{\frac{1-\theta}{\theta}}
	 \Big[\frac{1}{\sqrt{a}}\sqrt{\Phi(x^{\ell(k-m-1)})-\Phi(x^{\ell(k)})}\,\Big]^{\frac{1-\theta}{\theta}}\nonumber\\
  &=\frac{1}{2}\sum_{j=k}^{\infty}\Xi_j+\frac{bc}{2}
	 a^{\frac{\theta-1}{2\theta}}\big[bc(1-\theta)\big]^{\frac{1-\theta}{\theta}}\Big[\sum_{j=k-m-1}^{k-1}\Gamma_j^2\Big]^{\frac{1-\theta}{2\theta}}\nonumber\\
 &\le\frac{1}{2}\sum_{j=k}^{\infty}\Xi_j+\frac{bc}{2}
	 a^{\frac{\theta-1}{2\theta}}\big[bc(1-\theta)\big]^{\frac{1-\theta}{\theta}}\Big[\sum_{j=k-m-1}^{k-1}\Gamma_j\Big]^{\frac{1-\theta}{\theta}}.    
 \end{align}
 For each $k\in\mathbb{N}$, let $\widetilde{\Lambda}_k\!:=\sum_{j=k}^{\infty}\Gamma_{j}$ and write $\widehat{c}(\theta):=\frac{bc}{2}a^{\frac{\theta-1}{2\theta}}\big[bc(1\!-\!\theta)\big]^{\frac{1-\theta}{\theta}}$. When $\theta\in(0,1/2]$, since $\frac{1-\theta}{\theta}\ge 1$ and $\sum_{j=k-m-1}^{k-1}\Gamma_j\le 1$ for all $k\ge\widehat{k}$ (if necessary by increasing $\widehat{k}$), combining \eqref{KLrate-Gammak} and the first inequality in \eqref{XG_rate} leads to
 \[
 \widetilde{\Lambda}_k\le \frac{1}{2}\widehat{\gamma}\widehat{\varrho}^{k}+\widehat{c}(\theta)(\widetilde{\Lambda}_{k-m-1}-\widetilde{\Lambda}_{k}). 
 \]
 Set $\varrho_1\!:=\!\frac{\widehat{c}(\theta)}{1+\widehat{c}(\theta)}$ and $\gamma_1=\frac{\widehat{\gamma}}{2(1+\widehat{c}(\theta))}$. Then, it holds
 $\widetilde{\Lambda}_{k}\le\varrho_1\widetilde{\Lambda}_{k-m-1}+\gamma_1\widehat{\varrho}^{k}$ for all $k\ge\widehat{k}$. By using this recursion relation, for each $k\ge\widehat{k}$ and $j\in\{1,\ldots,\lfloor\frac{k-\widehat{k}}{m+1}\rfloor\}$, 
 \[
   \widetilde{\Lambda}_{k}\le\rho_1^{j}\widetilde{\Lambda}_{k-j(m+1)}+\gamma_1\sum_{l=1}^{j}\rho_1^{l-1}\widehat{\rho}^{k-(l-1)(m+1)}.
 \] 
 Then, by setting $\widetilde{k}\!:=\!\lfloor\frac{k-\widehat{k}}{m+1}\rfloor$, we obtain
 \begin{equation*}
  \widetilde{\Lambda}_{k}
  \le\varrho_1^{\widetilde{k}}\widetilde{\Lambda}_{k-\widetilde{k}(m+1)}+\gamma_1\widehat{\varrho}^{k}\Big[1+\frac{\varrho_1}{\widehat{\varrho}^{m+1}}+\cdots+\big(\frac{\varrho_1}{\widehat{\varrho}^{m+1}}\big)^{\widetilde{k}-1}\Big].
  \end{equation*}
  After a simple calculation for $\frac{\varrho_1}{\widehat{\varrho}^{m+1}}>1,\frac{\varrho_1}{\widehat{\varrho}^{m+1}}=1$ and $\frac{\varrho_1}{\widehat{\varrho}^{m+1}}<1$, respectively, there exist $\widetilde{\gamma}>0$ and $\widetilde{\varrho}\in(0,1)$ such that $\widetilde{\Lambda}_k\le\widetilde{\gamma}\widetilde{\varrho}^{k}$ for all $k$ large enough. When $\theta\in(1/2,1)$, from the above \eqref{KLrate-Gammak} and the first inequality in \eqref{XG_rate}, for all $k\ge\widehat{k}$, 
  \begin{align*}
  \widetilde{\Lambda}_k=\sum_{j=k}^{\infty}\Gamma_j&\le\frac{1}{2}\sum_{j=k}^{\infty}\Xi_j
	 +\widehat{c}(\theta)\Big[\sum_{j=k-m-1}^{k-1}\Gamma_j\Big]^{\frac{1-\theta}{\theta}}\\
  &\le\frac{1}{2}\widehat{\gamma}{k}^{\frac{1-\theta}{1-2\theta}}+\widehat{c}(\theta)(\widetilde{\Lambda}_{k-m-1}-\widetilde{\Lambda}_{k})^{\frac{1-\theta}{\theta}}\\
  &\le M_1\max\big\{(\widetilde{\Lambda}_{k-m-1}-\widetilde{\Lambda}_{k})^{\frac{1-\theta}{\theta}},{k}^{\frac{1-\theta}{1-2\theta}}\big\}
  \end{align*}  
  with $M_1=\frac{1}{2}\widehat{\gamma}+\widehat{c}(\theta)$. The result follows Lemma \ref{lemma2-sequence}. The proof is completed.  
 \end{proof}

Now combining Proposition~\ref{KL-Xirate} and Theorem~\ref{KL-converge} leads to the convergence rate of $\{x^k\}_{k\in\mathbb{N}}$. Among others, when $\theta=0$, the conclusion follows 
by using Lemma~\ref{lemma1-Phi} (iii) and a similar argument as those for \cite[Theorem 2.9~(i)]{QianPan23}. 
\begin{theorem}\label{KL-rate}
 Let $\{x^k\}_{k\in\mathbb{N}}$ and $\{z^k\}_{k\in\mathbb{N}}$ be the sequences satisfying H1-H4. Suppose that $\{z^k\}_{k\in\mathbb{N}}$ is bounded, and that $\Theta$ is a KL function of exponent $\theta$. Then, 
 \begin{description}
 \item [(i)] when $\theta=0$, the sequence $\{x^k\}_{k\in\mathbb{N}}$ converges to $\overline{x}$ within a finite number of steps; 
		
 \item [(ii)] when $\theta\in(0,1)$, there exist $\gamma>0,\varrho\in(0,1)$ such that for $k$ large enough,
  \begin{equation}\label{aim-ineq-rate}
  \|x^k-\overline{x}\|\le\sum_{j=k}^{\infty}\|x^{j+1}\!-\!x^j\|\le\left\{\begin{array}{cl}
   \gamma\varrho^{k} &{\rm if}\ \theta\in(0,1/2],\\
    \gamma k^{\frac{1-\theta}{1-2\theta}}&{\rm if}\ \theta\in(1/2,1).
  \end{array}\right.
  \end{equation}
  \end{description}
\end{theorem}
\subsection{Ensuring condition H4}\label{sec3.2}

 From Section \ref{sec3.1}, we see that the convergence analysis of the iterative framework H1-H4 involves handling the term $\sum_{j=\ell(k-1)}^{\ell(k)-1}\|x^{j+1}\!-\!x^{j}\|$ 
 for $\mathbb{N}\ni k>m$, and condition H4 is specially singled out to bound this term with $\Xi_k+\sum_{j=k-m-1}^{k-1}\Gamma_j$, as show by Proposition~\ref{prop-H4}. Condition H4 seems a little complicated and we need to clarify when it holds. Good news is that it can be satisfied under rather reasonable conditions. For example, when the difference of consecutive objective values can be bounded by $({\overline{b}}/{2})\|x^{k+1}-x^k\|^2$ for some $\overline{b}>0$, i.e., the objective value may increase but the increase magnitude is bounded by $({\overline{b}}/{2})\|x^{k+1}-x^k\|^2$, H4 automatically holds by the following lemma. 
\begin{lemma}\label{sequence-cond}
 If there exist $\overline{k}>m$ and $\overline{b}>0$ such that $\Phi(x^{k+1})-\Phi(x^k)\le ({\overline{b}}/{2})\|x^{k+1}\!-x^k\|^2$ for all $k\ge\overline{k}$, then condition H4 automatically holds.
\end{lemma}
\begin{proof}
 Fix any $k\ge\overline{k}$. For each $i\in[\ell(k\!-\!1)+1,\ell(k)\!-\!1]$, it holds
 \begin{align*}
 \Phi(x^{\ell(k)})-\Phi(x^i)=\sum_{j=i+1}^{\ell(k)}\big[\Phi(x^j)-\Phi(x^{j-1})\big]&\le\frac{\overline{b}}{2}\sum_{j=i+1}^{\ell(k)}\|x^{j}-x^{j-1}\|^2.
 \end{align*}
 Using $\sqrt{\alpha\!+\!\beta}\!\le\!\sqrt{\alpha}\!+\!\sqrt{\beta}$ for $\alpha,\beta\!\ge\! 0$ leads to H4 for $\mu=\sqrt{{\overline{b}}/{2}}$ and any $\tau\in(0,1)$.  
\end{proof}
 
 As will be shown in Section \ref{sec4}, the objective value sequence from the GLL-type nonmonotone line search PG methods there indeed satisfies the requirement of Lemma~\ref{sequence-cond}. Alternatively, if the objective value sequence does not satisfy such a condition but $\Theta=\Phi$ is suitable for H3, then H4 is actually a natural product of the prox-regularity of $\Phi$ on $\mathcal{P}_{\mathbb{X}}(Z^*)$. This result is implied by Proposition~\ref{prox-regular} and Proposition~\ref{prop-ass0} below. 
\begin{proposition}\label{prox-regular}
 Let $\{x^k\}_{k\in\mathbb{N}}$ and $\{z^k\}_{k\in\mathbb{N}}$ be the bounded sequences satisfying H1-H3 with $\Theta=\!\Phi$. If there is $x^*\in\mathcal{P}_{\mathbb{X}}(Z^*)$ for which there exist $\kappa>0$ and $\varepsilon>0$ such that for all $x,y\in\mathbb{B}(x^*,\varepsilon)\cap[\Phi<\Phi(x^*)+\varepsilon]$ with ${\rm dist}(0,\partial\Phi(x))+{\rm dist}(0,\partial\Phi(y))<\varepsilon$, 
 \begin{equation}\label{ass0-ineq}
  \Phi(x)-\Phi(y)\le \kappa\big[{\rm dist}(0,\partial\Phi(x))^2 + {\rm dist}(0,\partial\Phi(y)) \| x-y\|
          + \|x-y\|^2\big],
 \end{equation} 
 then condition H4 holds.
\end{proposition} 
\begin{proof}
 From $x^*\in\mathcal{P}_{\mathbb{X}}(Z^*)$, there exists $z^*\in Z^*$ such that $x^*=\mathcal{P}_{\mathbb{X}}(z^*)$. Since $z^*\in Z^*$, there exists $\mathcal{K}\subset\mathbb{N}$ such that $\lim_{\mathcal{K}\ni k\to\infty}z^k=z^*$, so that $\lim_{\mathcal{K}\ni k\to\infty}x^k=x^*$. By Proposition \ref{prop1-Phi}, we have $\lim_{\mathcal{K}\ni k\to\infty}x^{\ell(k)}=x^*$. From the proof of Proposition \ref{prop2-Phi} (i), it follows  $\Phi(x^*)=\varpi^*=\lim_{k\to\infty}\Phi(x^k)$, and there exists $k^*>m$ such that $\Phi(x^k)<\Phi(x^*)+\varepsilon$ for all $k\ge k^*$. Recall that $\sum_{k=1}^{\infty}\Gamma_k\!+\!\sum_{k=1}^{\infty}\Xi_{k}<\infty$ by Proposition \ref{prop2-Phi} (iii) and $\lim_{k\to\infty}(x^{k+1}\!-\!x^k)=0$ by Proposition \ref{prop1-Phi}. Therefore, for each $k\ge k^*$ (if necessary by increasing $k^*$), with $\widehat{\mu}:=c(\mu,\tau,a,m)$ it holds
 \begin{equation}\label{key-ineq31}
  \|x^{\ell(k^*)}\!-\!x^*\|+(b\!+\!1)\!\!\sum_{j=\ell(k-1)}^{\ell(k)-1}\!\|x^{j+1}\!-\!x^j\|+(m\!+\!1)\widehat{\mu}\bigg[\sum_{j=k^*-m-1}^{\infty}\!\!\Gamma_j+\sum_{j=k^*}^{\infty}\Xi_{j}\bigg]<\varepsilon.
 \end{equation} 
 We claim that the conclusion holds if for each $k\ge k^*$ and $i\!\in\![\ell(k\!-\!1)\!+\!1,\ell(k)\!-\!1]$, 
 \begin{equation}\label{ass-sequence}
 \Phi(x^{\ell(k)})-\Phi(x^i)\le\kappa\big[{\rm dist}(0,\partial\Phi(x^{\ell(k)}))^2+{\rm dist}(0,\partial\Phi(x^i))\|x^{\ell(k)}-x^i\|+\|x^{\ell(k)}-x^i\|^2\big].
 \end{equation}
 Indeed, fixing any $k\ge k^*$, for each $i\in[\ell(k\!-\!1)+1,\ell(k)-1]$, it holds 
 \begin{align}\label{key-ineq32}
 	\ \sqrt{\Phi(x^{\ell(k)})\!-\!\Phi(x^{i})}
 	&\stackrel{\eqref{ass-sequence}}{\le}\!\sqrt{\kappa\big[{\rm dist}(0,\partial \Phi(x^{\ell(k)}))^2+{\rm dist}(0,\partial \Phi(x^{i}))\|x^{\ell(k)}\!-\!x^i\|+\|x^{\ell(k)}\!-\!x^i\|^2\big]}\nonumber\\
 	&\le\sqrt{\kappa}\big[{\rm dist}(0,\partial \Phi(x^{\ell(k)}))+\!\sqrt{{\rm dist}(0,\partial \Phi(x^{i}))\|x^{\ell(k)}\!-\!x^i\|}+\|x^{\ell(k)}\!-\!x^i\|\big]\nonumber\\
 	&\stackrel{\eqref{dist-Theta}}{\le}\sqrt{\kappa}b\|x^{\ell(k)}\!-\!x^{\ell(k)-1}\|+\sqrt{\kappa b\|x^{i}\!-\!x^{i-1}\|\|x^{\ell(k)}\!-\!x^i\|}+\sqrt{\kappa}\|x^{\ell(k)}\!-\!x^i\|\nonumber\\
 	&\le\sqrt{\kappa}b\,\Xi_k+\tau\sqrt{a}\|x^{i}\!-\!x^{i-1}\|+ \frac{\kappa b}{4\tau\sqrt{a}}\|x^{\ell(k)}\!-\!x^i\|+\sqrt{\kappa}\|x^{\ell(k)}\!-\!x^i\|\nonumber\\
 	&=\sqrt{\kappa}b\,\Xi_k +\tau\sqrt{a}\|x^{i}\!-\!x^{i-1}\|+\frac{\kappa b+ 4\tau\sqrt{a\kappa}}{4\tau\sqrt{a}} \big\|\sum_{j=i+1}^{\ell(k)}[x^{j}-x^{j-1}]\big\|\nonumber\\
 	&\le \tau\sqrt{a}\|x^{i}\!-\!x^{i-1}\|+\frac{\kappa b+ 4\tau\sqrt{a\kappa}(b+1)}{4\tau\sqrt{a}} \sum_{j=i+1}^{\ell(k)}\|x^{j}-x^{j-1}\|,
 \end{align}
 where the penultimate inequality is got by using $\sqrt{\alpha\beta}\le\frac{1}{2}(\alpha+\beta)$ for $\alpha=2\tau\sqrt{a}\|x^{i}\!-\!x^{i-1}\|$ and $\beta=\frac{\kappa b}{2\tau\sqrt{a}}\|x^{\ell(k)}\!-\!x^i\|$, and the last one is due to $\Xi_k\le \sum_{j=i+1}^{\ell(k)}\|x^{j}-x^{j-1}\|$. The above inequality shows that condition H4 holds with $\mu=\frac{\kappa b+ 4\tau\sqrt{a\kappa}(b+1)}{4\tau\sqrt{a}}$. 
 
 Next we prove by induction that the claim \eqref{ass-sequence} holds for every $k\ge k^*$. We first consider $k=k^*$. The previous \eqref{key-ineq31} implies $x^i\in \mathbb{B}(x^*,\varepsilon)$ for each $i\in[\ell(k^*\!-\!1)+1,\ell(k^*)\!-\!1]$ by noting that 
 \begin{align*}
 \|x^{i}-x^*\|&\le\|x^{\ell(k^*)}\!-\!x^*\|
 +\sum_{j=i+1}^{\ell(k^*)}\!\!\|x^{j}\!-\!x^{j-1}\|
 \le\|x^{\ell(k^*)}\!-\!x^*\|
 +\!\!\sum_{j=\ell(k^*-1)+2}^{\ell(k^*)}\!\!\|x^{j}\!-\!x^{j-1}\|<\varepsilon.
 \end{align*}
 On the other hand, from H3 and $\Theta=\Phi$, for each $i\in[\ell(k^*\!-\!1)+1,\ell(k^*)\!-\!1]$, 
 \[
   {\rm dist}(0,\partial\Phi(x^{\ell(k^*)}))+{\rm dist}(0,\partial\Phi(x^{i}))\le b\big[\|x^{\ell(k^*)}-x^{\ell(k^*)-1}\|+\|x^{i}-x^{i-1}\|\big]\stackrel{\eqref{key-ineq31}}{<}\varepsilon.
 \]
 Recall that $\Phi(x^k)<\Phi(x^*)+\varepsilon$ for all $k\ge k^*$. Invoking inequality \eqref{ass0-ineq} with $x=x^{\ell(k^*)}$ and $z=x^{i}$ for all $i\in[\ell(k^*\!-\!1)+1,\ell(k^*)-\!1]$ shows that \eqref{ass-sequence} holds for $k=k^*$. Now assume that \eqref{ass-sequence} holds with $k=k^*,\ldots,\nu$ for some $\nu\ge k^*$. Then, from the above arguments, the inequality \eqref{key-ineq32} holds for $k=k^*,\ldots,\nu$, so condition H4 holds for $\overline{k}=k^*$ and $k=k^*,\ldots,\nu$. Thus, by Proposition \ref{prop-H4},  
 \begin{equation}\label{for-ineq32}
  \sum_{j=\ell(k^*)}^{\ell(\nu)-1}\|x^{j+1}-x^j\|\le(m\!+\!1)\widehat{\mu}\bigg[\sum_{j=k^*-m-1}^{\nu-1}\Gamma_j+\sum_{j=k^*}^{\nu}\Xi_{j}\bigg].
 \end{equation}
 Next we prove \eqref{ass-sequence} for $k=\nu+1$. Indeed, for each $i\in[\ell(\nu)+1,\ell(\nu\!+\!1)]$, 
 \begin{align*}\label{temp-ineq32}
 \|x^{i}-x^*\|&\le\|x^{\ell(k^*)}-x^*\|+\sum_{j=\ell(k^*)}^{\ell(\nu)-1}\|x^{j+1}-x^j\|+\sum_{j=\ell(\nu)}^{i-1}\!\!\|x^{j+1}-x^j\|\nonumber\\
 &\le\|x^{\ell(k^*)}-x^*\|+\sum_{j=\ell(k^*)}^{\ell(\nu)-1}\|x^{j+1}-x^j\|+\sum_{j=\ell(\nu)}^{\ell(\nu+1)-1}\!\!\|x^{j+1}-x^j\|\\
 &\stackrel{\eqref{for-ineq32}}{\le}\|x^{\ell(k^*)}\!-x^*\|+(m+1)\widehat{\mu}\bigg[\sum_{j=k^*-m-1}^{\nu-1}\!\!\Gamma_j+\sum_{j=k^*}^{\nu}\Xi_{j}\bigg]\\
 &\quad\ +\!\!\sum_{j=\ell(\nu)}^{\ell(\nu+1)-1}\!\!\|x^{j+1}\!-x^j\|\stackrel{\eqref{key-ineq31}}{<}\varepsilon.
 \end{align*}
 While from H3 and $\Theta=\Phi$, for each $i\in[\ell(\nu)+1,\ell(\nu\!+\!1)-1]$, it holds
 \[
   {\rm dist}(0,\partial\Phi(x^{\ell(\nu+1)}))+{\rm dist}(0,\partial\Phi(x^{i}))\le b\big[\|x^{\ell(\nu+1)}-x^{\ell(\nu+1)-1}\|+\|x^{i}-x^{i-1}\|\big],
 \]
 which along with \eqref{key-ineq31} implies ${\rm dist}(0,\partial\Phi(x^{\ell(\nu+1)}))+{\rm dist}(0,\partial\Phi(x^{i}))<\varepsilon$. Recall that $\Phi(x^{\nu+1})<\Phi(x^*)+\varepsilon$. Using \eqref{ass0-ineq} with $x=x^{\ell(\nu+1)}$ and $z=x^{i}$ for all $i\in[\ell(\nu)+1,\ell(\nu\!+\!1)\!-\!1]$ shows that \eqref{ass-sequence} holds for $k=\nu+1$. Thus, we finish the proof.  
 \end{proof}

 The assumption in Proposition \ref{prox-regular}, like condition H4, is complicated, but due to Proposition \ref{prop2-Phi} (ii) and $\Theta=\Phi$, surprisingly it is implied by the prox-regularity of $\Phi$ on $\mathcal{P}_{\mathbb{X}}(Z^*)$ by the following proposition. 
\begin{proposition}\label{prop-ass0}
 If $\Phi$ is prox-regular at $x^*$ for $0$, then there exist $\kappa>0$ and $\varepsilon>0$ such that \eqref{ass0-ineq} holds for any $x,y\in\mathbb{B}(x^*,\varepsilon)\cap[\Phi<\Phi(x^*)+\varepsilon]$ with ${\rm dist}(0,\partial\Phi(x))+{\rm dist}(0,\partial\Phi(y))<\varepsilon$. 
\end{proposition}
\begin{proof}
 Since $\Phi$ is prox-regular at $x^*$ for $0$, by Definition \ref{prox-reg-def}, $0\in\partial\Phi(x^*)$, and there exist $\varepsilon>0$ and $r>0$ such that, whenever $(x'',v'')\in\mathbb{B}((x^*,0),\varepsilon)\cap{\rm gph}\,\partial\Phi$ with $\Phi(x'')<\Phi(x^*)+\varepsilon$, 
 \begin{equation*}
 \Phi(x')\ge\Phi(x'')+\langle v'',x'-x''\rangle-({r}/{2})\|x'-x''\|^2\quad\ \forall x'\in\mathbb{B}(x^*,\varepsilon).
 \end{equation*}
 Pick any $x,y\in\mathbb{B}(x^*,\varepsilon)\cap[\Phi<\Phi(x^*)+\varepsilon]$ with ${\rm dist}(0,\partial\Phi(x))+{\rm dist}(0,\partial\Phi(y))<\varepsilon$. Then there exists $v\in\partial\Phi(x)$ such that $\|v\|={\rm dist}(0,\partial\Phi(x))\le\varepsilon$, so $(x,v)\in\mathbb{B}((x^*,0),\varepsilon)\cap{\rm gph}\,\partial\Phi$ with $\Phi(x)<\Phi(x^*)+\varepsilon$. Using the above inequality with $x'=y,x''=x, v''=v$ yields
 \begin{align*}
  \Phi(y)&\ge\Phi(x)+\langle v,y-x\rangle-\frac{r}{2}\|x-y\|^2\\
  &\ge\Phi(x) -{\rm dist}(0,\partial\Phi(x))\|y-x\|-\frac{r}{2}\|x-y\|^2\\
  &\ge\Phi(x)-\frac{1}{2}({\rm dist}(0,\partial\Phi(x)))^2-\frac{1+r}{2}\|x-y\|^2 \\
  & \ge \Phi(x) - \frac{1+r}2 \big[{\rm dist}(0,\partial\Phi(x))^2 + {\rm dist}(0,\partial\Phi(y)) \| x-y\|
  + \|x-y\|^2\big].
 \end{align*}
 This shows that the desired conclusion holds with $\kappa=(1\!+\!r)/2$ and $\varepsilon$.  
 \end{proof}  

\section{Applications}\label{sec4}

This section demonstrates that the proposed iterative framework H1-H4 encompasses some existing algorithms, so the full convergence of their iterate sequences can be achieved by leveraging the main results of Section \ref{sec3}. First, we take a closer look at the convergence of ${\rm NPG}_{\rm major}$ for the nonconvex and nonsmooth composite problem
\begin{equation}\label{DC-prob}
 \min_{x\in\mathbb{X}} F(x):=f(x)+g(x)-h(x),
 \end{equation}
 where $f,g\!:\mathbb{X}\to\overline{\mathbb{R}}$ and $h\!:\mathbb{X}\to\mathbb{R}$ satisfy the conditions stated in Assumption \ref{ass1}. 
 \begin{assumption}\label{ass1}
 {\bf(i)} $f$ is a proper and lsc function that is continuously differentiable on an open set $\mathcal{O}\supset{\rm dom}\,g\ne\emptyset$ with strictly continuous gradient $\nabla\!f$;  

 \noindent
 {\bf(ii)} $g$ is a proper and lsc function that is bounded from below;

 \noindent
 {\bf(iii)} $h$ is a convex function, and the function $F$ is bounded from below.  
 \end{assumption}
\subsection{Full convergence of \texorpdfstring{${\rm NPG}_{\rm major}$}{NPG\_major}}\label{sec4.1}

For model \eqref{DC-prob} with an $L$-smooth $f$, Liu et al. \cite{LiuPong19} proposed a GLL-type nonmonotone PG method with majorization (${\rm NPG}_{\rm major}$ for short), which has the same iteration steps as Algorithm \ref{NPG_major} below with $\alpha=0$ and $\partial(-h)(x^k)$ replaced by $-\partial h(x^k)$. To the best of our knowledge, the full convergence of its iterate sequence was still not established. We fill this gap by proving that the iterate sequence of Algorithm \ref{NPG_major} conforms to H1-H4. 
 \begin{algorithm}[h]
 \caption{\label{NPG_major}{\bf\,(${\rm NPG}_{\rm major}$ for solving \eqref{DC-prob})}}
 \begin{algorithmic}
 \State{Input: $m\in\mathbb{N},0<\gamma_{\rm min}\le\gamma_{\rm max}<\infty,\varrho>1,\delta\in(0,1),\alpha\in[0,1],c>0,x^0\in{\rm dom}\,g$. Set $k:=0$.}
\While{the termination condition is not satisfied}
\State{1. Choose $\gamma_{k,0}\in[\gamma_{\rm min},\gamma_{\rm max}]$}.

 \State{2. \textbf{For} $j=0,1,2,\ldots$}

 \State{\qquad\ Let $\gamma_{k,j}:=\varrho^j\gamma_{k,0}$ and pick any $\xi^k\in\partial(-h)(x^k)$. }
 
 \State{\qquad\ Compute the unique solution $x^{k,j}$ of the problem
  \begin{equation}\label{subprobk-dc}
     \min_{x\in\mathbb{X}}f(x^k)+\langle\nabla\!f(x^k)+\xi^k,x-x^k\rangle+\frac{\gamma_{k,j}}{2}\|x-x^k\|^2+g(x).
  \end{equation}}
  \State{\qquad\ If $F(x^{k,j})\le\max\limits_{[k-m]_{+}\le i\le k}\!F(x^{i})-\frac{\alpha\delta\gamma_{k,j}+(1-\alpha)c}{2}\|x^{k,j}-x^k\|^2$, then go to step 3.}
   \State{\quad\ \textbf{End}}
  
   \State{3. Let $j_k:=j,\gamma_{k}:=\gamma_{k,j_k}$ and $x^{k+1}:=x^{k,j_k}$. Set $k\leftarrow k+1$ and go to Step 1.}
 \EndWhile 
 \end{algorithmic}
 \end{algorithm}
 \begin{lemma}\label{well-defined1}
 For each $k\in\mathbb{N}$ with $x^k\ne x^{k-1}$, the inner loop of Algorithm \ref{NPG_major} stops within a finite number of steps. If there exists some $k\in\mathbb{N}$ such that $x^k=x^{k-1}$, then $x^{k}$ is a stationary point of \eqref{DC-prob}, i.e.,
 $x^k\in\mathcal{S}^*\!:=\big\{x\in\mathbb{X}\,|\,0\in\nabla\!f(x)+\partial g(x)+\partial(-h)(x)\big\}$. 
 \end{lemma}
 \begin{proof}
 Suppose on the contrary that there exists some $k\in\mathbb{N}$ with $x^k\ne x^{k-1}$ such that the inner loop does not stop within a finite number of steps, i.e., 
 \begin{equation}\label{aim-ineq30}
  F(x^{k,j})>F(x^{\ell(k)})-\frac{\alpha\delta\gamma_{k,j}+(1\!-\!\alpha)c}{2}\|x^{k,j}-x^k\|^2\quad\forall j\in\mathbb{N}.
 \end{equation}
 Clearly, $\lim_{j\to\infty}\gamma_{k,j}=\infty$. For each $j\in\mathbb{N}$, from the optimality of $x^{k,j}$ to problem \eqref{subprobk-dc}, 
 \begin{equation}\label{f-ineq0}
  f(x^k)+\langle\nabla\!f(x^k)+\xi^k,x^{k,j}-x^k\rangle+\frac{\gamma_{k,j}}{2}\|x^{k,j}-x^k\|^2+g(x^{k,j})\le F(x^k)+h(x^k).
 \end{equation}
 In view of Assumption \ref{ass1} (ii),  $g(x^{k,j})\ge\inf_{x\in\mathbb{X}}g(x)>-\infty$ for all $j\in\mathbb{N}$. From \eqref{f-ineq0} and $\lim_{j\to\infty}\gamma_{k,j}=\infty$, we deduce $\lim_{j\to\infty}\|x^{k,j}-x^{k}\|=0$. Along with the strict continuity of $\nabla\!f$ at $x^k$ and the descent lemma in \cite[Lemma 5.7]{Beck17}, there exist $L_k>0$ and $\overline{j}\in\mathbb{N}$ such that 
 \begin{equation}\label{fdescent}
 f(x^{k,j})-f(x^k)-\langle\nabla\!f(x^k),x^{k,j}-x^k\rangle\le\frac{L_k}{2}\|x^{k,j}-x^k\|^2\quad\forall j\ge\overline{j}.
 \end{equation}
 On the other hand, for each $j\ge\overline{j}$, the above \eqref{f-ineq0} can be rearranged as 
 \begin{align*}
 F(x^{k,j})+\frac{\gamma_{k,j}}{2}\|x^{k,j}-x^k\|^2
 &\le f(x^{k,j})-f(x^k)-\langle\nabla\!f(x^k)+\xi^k,x^{k,j}-x^k\rangle\\
 &\quad\ +F(x^k)+h(x^k)-h(x^{k,j}).
 \end{align*}
 Recall that $\xi^k\in\partial(-h)(x^k)\subset-\partial h(x^k)$. We have $h(x^{k,j})-h(x^k)\ge\langle-\xi^k,x^{k,j}-x^k\rangle$ for all $j\in\mathbb{N}$. Together with the above two inequalities, for each $j\ge\overline{j}$, it holds
 \begin{equation*}
  F(x^{k,j})+\frac{\gamma_{k,j}}{2}\|x^{k,j}\!-\!x^k\|^2
  \le \frac{L_k}{2}\|x^{k,j}\!-\!x^k\|^2+F(x^k)\le \frac{L_k}{2}\|x^{k,j}\!-\!x^k\|^2+F(x^{\ell(k)}).
 \end{equation*}
 Combining this inequality with the above \eqref{aim-ineq30}, for any $j\ge\overline{j}$, we have 
 \[
   \frac{1}{2}[(1-\alpha\delta)\gamma_{k,j}-(1-\alpha)c-L_k]\|x^{k,j}\!-\!x^k\|^2<0.
 \]
 This is impossible because $\lim_{j\to\infty}\gamma_{k,j}=\infty$, $\delta\in(0,1)$ and $\alpha\in[0,1]$. The first part of the conclusions follows. If there is an index $k$ such that $x^{k}=x^{k-1}$, by the optimality of $x^k$ to \eqref{subprobk-dc}, $0\in\nabla\!f(x^{k-1})+\partial g(x^{k})+\partial(-h)(x^{k-1})=\nabla\!f(x^{k})+\partial g(x^{k})+\partial(-h)(x^{k})$.  
 \end{proof}

 Lemma \ref{well-defined1} implies the well-definedness of Algorithm \ref{NPG_major}. Let $\{x^k\}_{k\in\mathbb{N}}$ be the sequence produced by Algorithm \ref{NPG_major}. For each $k\in\mathbb{N}$, write $z^k:=(x^k,\xi^{k-1})$. To demonstrate that $\{x^k\}_{k\in\mathbb{N}}$ and $\{z^k\}_{k\in\mathbb{N}}$ conform to H1-H4, we need their boundedness stated below.    
\begin{lemma}\label{gammak-bound}
 Suppose that $\mathcal{L}_{F}(x^0)\!:=\{x\in\!\mathbb{X}\,|\,F(x)\le F(x^0)\}$ is bounded. Then, the sequences $\{z^k\}_{k\in\mathbb{N}}$ and $\{\gamma_k\}_{k\in\mathbb{N}}$ are bounded, so there exists $\gamma_*\!>\!0$ such that $\gamma_k\!\le\!\gamma_*$ for all $k\in\mathbb{N}$.
\end{lemma}
\begin{proof} 
 From the definition of $\ell(k)$ and step 2 of Algorithm \ref{NPG_major}, for each $k\in\mathbb{N}$, it holds 
 \begin{align*}
  F(x^{\ell(k+1)})&\le\max\big\{F(x^{k+1}),F(x^{\ell(k)})\}\\
   &\le\max\{F(x^{\ell(k)})-\frac{\alpha\delta\gamma_{k,j}+(1\!-\!\alpha)c}{2}\|x^{k,j}-x^k\|^2,F(x^{\ell(k)})\big\}
   \le F(x^{\ell(k)}),
 \end{align*}
 and hence $F(x^{k+1})\le F(x^{\ell(k)})\le\cdots\le F(x^0)$. This means that $\{x^k\}_{k\in\mathbb{N}}\subset\mathcal{L}_{F}(x^0)$, so  it is bounded due to the given assumption on $\mathcal{L}_{F}(x^0)$. Note that the mapping $\partial(-h)\!:\mathbb{X}\rightrightarrows\mathbb{X}$ is locally bounded by the finite convexity of $h$ and \cite[Theorem 9.13]{RW98}. Then, the sequence $\{\xi^k\}_{k\in\mathbb{N}}$ is bounded by \cite[Proposition 5.15]{RW98}, and the boundedness of $\{z^k\}_{k\in\mathbb{N}}$ follows. The rest only needs to prove that $\{\gamma_k\}_{k\in\mathbb{N}}$ is bounded. If not, there exists an index set $\mathcal{K}\subset\mathbb{N}$ such that $\lim_{\mathcal{K}\ni k\to\infty}\gamma_k=\infty$. For each $k\in\mathcal{K}$, let $\widehat{\gamma}_k\!\!:=\!\!\varrho^{-1}\gamma_k\!=\!\gamma_{k,j_k-1}$. Apparently, $\lim_{\mathcal{K}\ni k\to\infty}\!\widehat{\gamma}_k\!\!=\!\infty$. From step 2 of Algorithm \ref{NPG_major}, it follows 
 \begin{equation}\label{F-ineq40}
  F(x^{k,j_k-1})>F(x^{\ell(k)})-\frac{\alpha\delta\widehat{\gamma}_{k}+(1\!-\!\alpha)c}{2}\|x^{k,j_{k}-1}-x^k\|^2\quad\forall k\in\mathcal{K}. 
 \end{equation}
 Since $g(x^{k,j_k-1})\ge\inf_{x\in\mathbb{X}}g(x)>-\infty$, using $\lim_{\mathcal{K}\ni k\to\infty}\widehat{\gamma}_k=\infty$ and the boundedness of $\{z^k\}_{k\in\mathbb{N}}$, we deduce from \eqref{f-ineq0} for $k\in\mathcal{K}$ and $j=j_k\!-\!1$ that $\lim_{\mathcal{K}\ni k\to\infty}\|x^{k,j_k-1}\!-x^k\|=0$. Then, the boundedness of $\{x^k\}_{k\in\mathbb{N}}$ implies that of $\{x^{k,j_k-1}\}_{k\in\mathcal{K}}\subset{\rm dom}\,g\subset\mathcal{O}$. Thus, there exists a compact set $\Gamma\subset\mathcal{O}$ such that $x^{k,j_k-1}\in\Gamma$ and $x^k\in\Gamma$ for all $k\in\mathcal{K}$. Recall that $\nabla\!f$ is strictly continuous on $\mathcal{O}$ by Assumption \ref{ass1} (i), so it is Lipschitz continuous on $\Gamma$. Denote by $\widehat{L}_{\!f}$ the Lipschitz constant of $\nabla\!f$ on $\Gamma$. By the descent lemma in \cite[Lemma 5.7]{Beck17}, for each $k\in\mathcal{K}$, 
 \[
   f(x^{k,j_k-1})-f(x^k)-\langle\nabla\!f(x^k),x^{k,j_k-1}-x^k\rangle\le \frac{\widehat{L}_{\!f}}{2}\|x^{k,j_k-1}-x^k\|^2.
 \] 
 Henceforth, following the same argument as for Lemma \ref{well-defined1} and using the above \eqref{F-ineq40}, we obtain 
 \[
  \frac{(1-\alpha\delta)\widehat{\gamma}_{k}-(1-\alpha)c-\widehat{L}_{\!f}}{2}\|x^{k,j_{k}-1}-x^k\|^2<0\quad\forall k\in\mathcal{K}, 
 \]
 which is impossible by recalling $\lim_{\mathcal{K}\ni k\to\infty}\!\widehat{\gamma}_k\!\!=\!\infty,\,\alpha\in[0,1]$ and $\delta\in(0,1)$.  
 \end{proof}

 Now we are in a position to establish the full convergence of $\{x^k\}_{k\in\mathbb{N}}$ by proving that $\{x^k\}_{k\in\mathbb{N}}$ and $\{z^k\}_{k\in\mathbb{N}}$ armed with the following $\Theta$ conform to H1-H4: 
 \begin{equation}\label{def-Theta}
 \Theta(z):=f(x)+g(x)+\langle s,x\rangle+h^*(-s)\quad\forall z=(x,s)\in\mathbb{Z}:=\mathbb{X}\times\mathbb{X},  
 \end{equation} 
 where $h^*$ denotes the conjugate of $h$, i.e., $h^*(x)=\sup_{u\in\mathbb{X}}\{\langle u,x\rangle-h(u)\}$.
 \begin{theorem}\label{theorem41}
 Suppose that the function $\Theta$ defined by \eqref{def-Theta} is a KL function, and that the set $\mathcal{L}_{F}(x^0)$ is bounded. Then, $\{x^k\}_{k\in\mathbb{N}}$ converges to a point $\overline{x}\in\mathcal{S}^*$. If $\Theta$ is a KL function of exponent $\theta\in(0,1)$, there exist $\gamma>0$ and $\varrho\in(0,1)$ such that \eqref{aim-ineq-rate} holds for all $k$ large enough. 
 \end{theorem} 
 \begin{proof}
 In view of Theorems \ref{KL-converge} and \ref{KL-rate}, it suffices to prove that $\{x^k\}_{k\in\mathbb{N}}$ and $\{z^k\}_{k\in\mathbb{N}}$ armed with the $\Theta$ in \eqref{def-Theta} conform to H1-H4. We complete the proof by the following three steps.

 \noindent
 {\bf Step 1: to prove that $\{x^k\}_{k\in\mathbb{N}}$ satisfies H1-H2.} From step 2 of Algorithm \ref{NPG_major}, $\{x^k\}_{k\in\mathbb{N}}$ satisfies H1 for $\Phi\!=\!F$ and $a\!=\!\frac{1}{2}[\alpha\delta\gamma_{\min}+(1\!-\!\alpha)c]$. 
 Then $\lim_{k\to\infty}\|x^{\ell(k)}-x^{\ell(k)-1}\|=0$ by Lemma \ref{lemma1-Phi} (ii). To prove that H2 holds, let $\{x^{\ell(k_q)}\}_{q\in\mathbb{N}}$ be a convergent subsequence with $x^*=\lim_{q\to\infty}x^{\ell(k_q)}$. Apparently, $x^*=\lim_{q\to\infty}x^{\ell(k_q)-1}$. For each $q\in\mathbb{N}$, from the optimality condition of $x^{\ell(k_q)}$ to problem \eqref{subprobk-dc} for $k=\ell(k_q)-1$, it follows 
  \begin{align*}  &\langle\nabla\!f(x^{\ell(k_q)-1})\!+\!\xi^{\ell(k_q)-1},x^{\ell(k_q)}\!-\!x^{\ell(k_q)-1}\rangle\!+\!\frac{\gamma_{\ell(k_q)-1}}{2}\|x^{\ell(k_q)}\!-\!x^{\ell(k_q)-1}\|^2\!+\!g(x^{\ell(k_q)})\nonumber\\
  &\le\langle\nabla\!f(x^{\ell(k_q)-1})+\xi^{\ell(k_q)-1},x^*-x^{\ell(k_q)-1}\rangle+\frac{\gamma_{\ell(k_q)-1}}{2}\|x^*\!-\!x^{\ell(k_q)-1}\|^2+g(x^*). 
 \end{align*}
 Passing the limit $q\to\infty$ to the above inequality and using $\lim_{q\to\infty}\|x^{\ell(k_q)}-x^{\ell(k_q)-1}\|=0$ and Lemma \ref{gammak-bound} leads to 
  $\limsup_{q\to\infty}g(x^{\ell(k_q)})\le g(x^*)$. Along with the continuity of $f$ and $h$, it holds that
  $\limsup_{q\to\infty}F(x^{\ell(k_q)})\le F(x^*)=F(\lim_{q\to\infty}x^{\ell(k_q)})$, so H2 holds. 

 \noindent
 {\bf Step 2: to prove that $\{z^k\}_{k\in\mathbb{N}}$ satisfies H3.} By the definition of conjugate functions, $h^*(-s)+\langle s,x\rangle\ge -h(x)$ for all $(x,s)\in\mathbb{X}\times\mathbb{X}$. Together with the expression of $\Theta$ in \eqref{def-Theta}, 
 \[
  \Theta(x,s)\ge F(x)\quad\ {\rm for\ all}\ (x,s)\in\mathbb{X}\times\mathbb{X}. 
 \]
 Thus, it suffices to prove \eqref{Phi-Theta} and \eqref{dist-Theta} for $\Phi=F$ and $\Theta$ in \eqref{def-Theta}. Fix any $k\in\mathbb{N}$. Note that
  \begin{align}\label{temp-Fineq1}
   F(x^{k+1})&\le f(x^{k+1})+g(x^{k+1})-h(x^k)+\langle\xi^k,x^{k+1}-x^k\rangle\nonumber\\
   &=f(x^{k+1})+g(x^{k+1})+h^*(-\xi^k)+\langle\xi^k,x^{k+1}\rangle=\Theta(z^{k+1}),
  \end{align}
  where the inequality is due to $-\xi^k\in\partial h(x^k)$ and the convexity of $h_2$, and the first equality follows \cite[Theorem 23.5]{Roc70}. Since $\mathcal{L}_{F}(x^0)$ is compact, from Assumption \ref{ass1} (i) we conclude that $\nabla\!f$ is Lipschitz continuous on $\mathcal{L}_{F}(x^0)$. Denote by $L_{\!f}$ the Lipschitz constant of $\nabla\!f$ on $\mathcal{L}_{F}(x^0)$. From $\{x^k\}_{k\in\mathbb{N}}\subset\mathcal{L}_{F}(x^0)$ and the descent lemma,
  \begin{align}\label{ThetaFk}
   \Theta(z^{k+1})&\stackrel{\eqref{temp-Fineq1}}{=}f(x^{k+1})+g(x^{k+1})-h(x^k)+\langle\xi^k,x^{k+1}-x^k\rangle\nonumber\\
   &\le f(x^k)+\langle \nabla f(x^k),x^{k+1}-x^k\rangle+\frac{L_f}{2}\|x^{k+1}-x^k\|^2\nonumber\\
   &\quad\ +\langle\xi^k,x^{k+1}-x^k\rangle+g(x^{k+1})-h(x^k)\nonumber\\
   &\le F(x^k)\!+\!\frac{L_f-\gamma_k}{2}\|x^{k+1}-x^k\|^2
   \!\le\! F(x^{\ell(k)})\!+\!\frac{L_f}{2}\|x^{k+1}-x^k\|^2,
  \end{align}
 where the second inequality is due to the optimality of $x^{k+1}$ to \eqref{subprobk-dc}. Combining \eqref{ThetaFk} with \eqref{temp-Fineq1}, we conclude that \eqref{Phi-Theta} holds with $\varsigma_k=\frac{L_{\!f}}{2}\|x^{k+1}-x^k\|^2$, where $\lim_{k\to\infty}\varsigma_k=0$ follows Proposition \ref{prop1-Phi}. On the other hand, from the optimality condition of $x^{k}$ to \eqref{subprobk-dc}, it follows
  \begin{equation}\label{DCxkoptcond}
   0\in\nabla\!f(x^{k-1})+\xi^{k-1}+\gamma_{k-1}(x^{k}-x^{k-1})+\partial g(x^{k}).
  \end{equation}
  Together with the expression of $\Theta$, it holds  
  \begin{align*}
  &\begin{pmatrix}
   \nabla\!f(x^k)-\nabla\!f(x^{k-1})-\gamma_{k-1}(x^k-x^{k-1})\\
    x^{k} -x^{k-1}
    \end{pmatrix}\in\partial\Theta(z^k)
  \end{align*}
  which, by the Lipschitz continuity of $\nabla\!f$ on $\mathcal{L}_{F}(x^0)$, implies \eqref{dist-Theta} for $b=1+L_{\!f}+\gamma_*$. 

 \noindent
 {\bf Step 3: to prove that $\{x^k\}_{k\in\mathbb{N}}$ satisfies H4.} The second inequality in \eqref{ThetaFk} along with \eqref{temp-Fineq1} implies that $F(x^{k+1})\le F(x^k)+\frac{L_{f}}{2}\|x^{k+1}-x^k\|^2$ for all $k\in\mathbb{N}$. The sequence $\{F(x^k)\}_{k\in\mathbb{N}}$ then satisfies the assumption of Lemma \ref{sequence-cond} with $\overline{b}=L_{\!f}$, and condition H4 holds by Lemma \ref{sequence-cond}. 
 
 Now the conclusions hold with $\overline{x}\in\mathcal{P}_{\mathbb{X}}(Z^*)$. To achieve the inclusion $\overline{x}\in\mathcal{S}^*$, we first claim that $\lim_{k\to\infty}g(x^k)=g(\overline{x})$. Indeed, from the optimality condition of $x^{k}$ to \eqref{subprobk-dc}, 
 \begin{align*}  &\langle\nabla\!f(x^{k-1})+\xi^{k-1},x^{k}\!-\!x^{k-1}\rangle+\frac{\gamma_{k-1}}{2}\|x^{k}\!-\!x^{k-1}\|^2+g(x^{k})\nonumber\\
  &\le\langle\nabla\!f(x^{k-1})+\xi^{k-1},\overline{x}-x^{k-1}\rangle+\frac{\gamma_{k-1}}{2}\|\overline{x}\!-\!x^{k-1}\|^2+g(\overline{x}). 
 \end{align*} 
 Passing the limit $k\to\infty$ to this inequality leads to $\limsup_{k\to\infty}g(x^k)\le g(\overline{x})$. Along with the lower semicontinuity of $g$, the claimed $\lim_{k\to\infty}g(x^k)=g(\overline{x})$ holds. In view of \cite[Proposition 8.7]{RW98}, the mapping $\partial g$ is outer semicontinuous.
 From $\overline{x}\in\mathcal{P}_{\mathbb{X}}(Z^*)$, there exists $\overline{\xi}\in\mathbb{X}$ such that $\overline{z}=(\overline{x},\overline{\xi})\in Z^*$, so there exists $\mathcal{K}\subset\mathbb{N}$ such that $\lim_{\mathcal{K}\ni k\to\infty}z^k=\overline{z}$. Clearly, $\overline{\xi}=\lim_{\mathcal{K}\ni k\to\infty}\xi^{k-1}$. Recall that $\xi^{k-1}\in\partial(-h)(x^{k-1})$ for each $k\in\mathcal{K}$. The outer semicontinuity of $\partial(-h)$ implies $\overline{\xi}\in\partial(-h)(\overline{x})$.
 Passing the limit $\mathcal{K}\ni k\to\infty$ to \eqref{DCxkoptcond} and using Lemma \ref{gammak-bound} and the outer semicontinuity of $\partial g$ leads to $0\in\nabla\!f(\overline{x})+\partial(-h)(\overline{x})+\partial g(\overline{x})$, i.e., $\overline{x}\in\mathcal{S}^*$.  
\end{proof}
 \begin{remark}\label{DC_H1_4}
 The parameter $\alpha\in[0,1]$ is introduced into Algorithm \ref{NPG_major} just for recovering the original ${\rm NPG}_{\rm major}$. When $\xi^k\in\partial(-h)(x^k)$ in step 2 is replaced with $\xi^k\in -\partial h(x^k)$, from the proofs of Lemmas \ref{well-defined1}-\ref{gammak-bound} and Theorem \ref{theorem41}, we notice that their conclusions still hold except that $\overline{x}\in S^*$ in Theorem \ref{theorem41} is replaced by $\overline{x}\in \widehat{S}^*:=\big\{x\in\mathbb{X}\,|\,0\in\nabla\!f(x)+\partial g(x)-\partial h(x)\big\}$. Thus, we provide the full convergence certificate for ${\rm NPG}_{\rm major}$ proposed by Liu et al. \cite{LiuPong19}.   
 \end{remark}
 \subsection{Full convergence of PGenls and PGnls}\label{sec4.2}
  
 We focus on the convergence of a GLL-type nonmonotone line search PG method with extrapolation (PGenls) for solving  \eqref{DC-prob} with $h\equiv 0$. Its iteration steps are described in Algorithm \ref{PGenls} where, for a given $\delta\ge 0$, the function $F_{\delta}$ is defined as
 \begin{equation}\label{Psidelta}
 F_{\delta}(z):=F(x)+({\delta}/{2})\|x-u\|^2
 \quad{\rm for}\ z:=(x,u)\in\mathbb{X}\times\mathbb{X}.
 \end{equation}
 Algorithm \ref{PGenls} with $\delta=0$ and $\beta_{\rm max}=0$ is precisely the GLL-type nonmonotone line-search PG method (PGnls) proposed in \cite{Kanzow22}. Notice that Algorithm \ref{PGenls} is different from the PGenls proposed by Yang \cite{Yang24} for solving \eqref{DC-prob} with $h\equiv 0$ and an $L_{\!f}$-smooth $f$, since the latter also requires $\gamma_{k,j}$ to be a variable of $F_{\delta}$ so that the stop condition of the inner loop involves an additional term $\frac{\gamma_{\ell(k)-1}}{4}\|x^{\ell(k)}-x^{\ell(k)-1}\|^2$. For Yang's PGenls,  only the convergence rate of the objective value sequence was achieved for the monotone case and the case without extrapolation by assuming that $F$ is a KL function of exponent $\theta$ and continuous relative to its domain. 
 \begin{algorithm}[h]
 \caption{\label{PGenls}{\bf\,(Nonmonotone line search PG with extrapolation)}}
 \begin{algorithmic} 
 \State{Input: $m\in\!\mathbb{N},\delta\ge 0,\alpha\in(0,1),0<\!\gamma_{\rm min}\le\gamma_{\rm max}<\infty,
 \beta_{\rm max}\in[0,1]$, $\varrho>1$ and $\nu\in(0,1/\varrho)$.
 Choose $x^0\in{\rm dom}\,g$. Let $x^{-1}=x^0,z^0=(x^0,x^{-1})$ and set $k:=0$.}
 \While{ the termination condition is not satisfied}
 \State{1. Choose $\beta_{k,0}\in[0,\beta_{\rm max}]$ and $\gamma_{k,0}\in[\gamma_{\rm min},\gamma_{\rm max}]$.}

 \State{2. \textbf{For} $j=0,1,2,\ldots$ \textbf{do}}

 \State{\qquad\ Let $\beta_{k,j}=\beta_{k,0}\nu^{j},\,\gamma_{k,j}=\gamma_{k,0}\varrho^{j}$ and
              $y^{k,j}=x^k+\beta_{k,j}(x^k\!-\!x^{k-1})$.}

 \State{\qquad\  Seek an optimal solution $x^{k,j}$ of the optimization problem
  \begin{equation}\label{subprobk-PG}
  \min_{x\in\mathbb{X}}f(y^{k,j})+\langle\nabla\!f(y^{k,j}),x-y^{k,j}\rangle+\frac{\gamma_{k,j}}{2}\|x-y^{k,j}\|^2+g(x).
  \end{equation}}
 \State{\qquad\ If $F_{\delta}(z^{k,j})\le\max\limits_{[k-m]_{+}\le i\le k} F_{\delta}(z^{i})\!-\!\frac{\alpha\gamma_{k,j}}{2}\|x^{k,j}\!-\!x^k\|^2-\frac{\alpha\delta}{2}\|x^{k}\!-\!x^{k-1}\|^2$ with}
 \State{\qquad\ $z^{k,j}\!:=(x^{k,j},x^k)$, then go to step 3.} 
  
  \State{\quad\ \textbf{End}}

  \State{3. Let $j_k\!:=j,\beta_{k}\!:=\beta_{k,j_k},\gamma_{k}\!:=\gamma_{k,j_k}, x^{k+1}\!:=x^{k,j_k},y^{k+1}\!:=y^{k,j_k},z^{k+1}\!:=z^{k,j_k}$.}

  \State{4. Set $k\leftarrow k+1$ and go to step 1.}
 \EndWhile
 \end{algorithmic}
 \end{algorithm}
\begin{lemma}\label{well-defined2}
 For each $k\in\mathbb{N}$ with $x^k\!\ne\!x^{k-1}$, the inner loop of Algorithm \ref{PGenls} stops within a finite number of steps, and consequently Algorithm \ref{PGenls} is well defined.  
 \end{lemma}
 \begin{proof}
 Suppose on the contrary that there exists some $k\in\mathbb{N}$ with $x^k\ne x^{k-1}$ such that the inner loop does not stop within a finite number of steps, i.e.,
 \begin{equation}\label{aim-ineq41}
 F_{\delta}(z^{k,j})>F_{\delta}(z^{\ell(k)})-\frac{\alpha\gamma_{k,j}}{2}\|x^{k,j}-x^k\|^2-\frac{\alpha\delta}{2}\|x^{k}-x^{k-1}\|^2\quad\ \forall j\in\mathbb{N}.
 \end{equation}
 Clearly, $\lim_{j\to\infty}\beta_{k,j}=0$ and $\lim_{j\to\infty}\gamma_{k,j}=\infty$. Furthermore, from $\nu\in(0,1/\varrho)$, we have $\lim_{j\to\infty}\gamma_{k,j}\beta_{k,j}=0$. For each $j\in\mathbb{N}$, from the optimality of $x^{k,j}$ to \eqref{subprobk-PG}, it holds 
 \begin{align}\label{optim-xkj}
  &\quad\langle\nabla\!f(y^{k,j}),x^{k,j}-y^{k,j}\rangle+\frac{\gamma_{k,j}}{2}\|x^{k,j}-y^{k,j}\|^2+g(x^{k,j})\nonumber\\
  &\le \langle\nabla\!f(y^{k,j}),x^{k}-y^{k,j}\rangle+\frac{\gamma_{k,j}}{2}\|x^{k}-y^{k,j}\|^2+g(x^{k})\nonumber\\
  &=\langle\nabla\!f(y^{k,j}),x^{k}-y^{k,j}\rangle+\frac{\gamma_{k,j}\beta_{k,j}^2}{2}\|x^{k}-x^{k-1}\|^2+g(x^{k}),
 \end{align}
 which, along with the expression of $F$, can be rearranged as follows 
 \begin{align}\label{Psidkval}
  F(x^{k,j})-F(x^{k})&\le f(x^{k,j})-f(x^k)-\langle\nabla\!f(x^k),x^{k,j}-x^{k}\rangle+\frac{\gamma_{k,j}}{2}\|x^{k}-y^{k,j}\|^2\nonumber\\
  &\quad +\langle\nabla\!f(x^{k})-\nabla\!f(y^{k,j}),x^{k,j}-x^{k}\rangle-\frac{\gamma_{k,j}}{2}\|x^{k,j}-y^{k,j}\|^2.
 \end{align} 
 Recall that $y^{k,j}=x^k+\beta_{k,j}(x^k\!-\!x^{k-1})$ and $\lim_{j\to\infty}\beta_{k,j}=0$, so $\lim_{j\to\infty}y^{k,j}=x^k$. 
 Since $g$ is bounded from below, using $\lim_{j\to\infty}\gamma_{k,j}=\infty$ and $\lim_{j\to\infty}\beta_{k,j}=\lim_{j\to\infty}\gamma_{k,j}\beta_{k,j}=0$, we deduce from \eqref{optim-xkj} that
 $\lim_{j\to\infty}\|x^{k,j}-y^{k,j}\|=0$. Along with $\lim_{j\to\infty}y^{k,j}=x^k$, it holds that 
 $\lim_{j\to\infty}x^{k,j}=x^{k}$.
 Since $\nabla\!f$ is strictly continuous at $x^k$, there exists $L_k>0$ such that for all sufficiently large $j$, 
 \begin{align*}
 &\|\nabla f(x^k)-\nabla f(y^{k,j})\|\le L_k\|x^k-y^{k,j}\|,\\
 &f(x^{k,j})-f(x^{k})-\langle\nabla\!f(x^{k}),x^{k,j}-x^{k}\rangle\le \frac{L_k}{2}\|x^{k,j}-x^{k}\|^2.
 \end{align*}
 Substituting the above two inequalities into  \eqref{Psidkval}, for all sufficiently large $j$, it holds
 \begin{align*}
 F(x^{k,j})-F(x^k)&\le\frac{L_k}{2}\|x^{k,j}-x^{k}\|^2+L_k\|x^k\!-\!y^{k,j}\|\|x^{k,j}\!-\!x^k\|\\
 &\quad -\frac{\gamma_{k,j}}{2}\big[\|x^{k,j}-y^{k,j}\|^2-\|x^k-y^{k,j}\|^2\big]\\
 &=\frac{L_k}{2}\|x^{k,j}-x^{k}\|^2+L_k\|x^k\!-\!y^{k,j}\|\|x^{k,j}\!-\!x^k\|\\
 &\quad -\frac{\gamma_{k,j}}{2}\big[\|x^{k,j}-x^{k}\|^2+2\langle x^{k,j}-x^k,x^k-y^{k,j}\rangle\big]\\
 &\le \frac{L_k-\gamma_{k,j}}{2}\|x^{k,j}\!-\!x^{k}\|^2+(L_k\!+\!\gamma_{k,j})\|x^{k}\!-\!x^{k,j}\|\|x^{k}\!-\!y^{k,j}\|\\
 &\le -\frac{\gamma_{k,j}\!-\!1\!-\!L_k}{2}\|x^{k,j}\!-\!x^{k}\|^2+(L_k\!+\!\gamma_{k,j})^2\|x^{k}\!-\!y^{k,j}\|^2\\
 &=-\frac{\gamma_{k,j}\!-\!1\!-\!L_k}{2}\|x^{k,j}\!-\!x^{k}\|^2+[(L_k\!+\!\gamma_{k,j})\beta_{k,j}]^2\|x^{k}\!-\!x^{k-1}\|^2,
 \end{align*}
 where the third inequality is obtained by using $\|u\|\|v\|\le\frac{1}{2}\|u\|^2+\|v\|^2$ for $u=x^k-x^{k,j}$ and $v=(L_k\!+\!\gamma_{k,j})(x^k-y^{k,j})$. 
 Then, from the expression of $F_{\delta}$, for  sufficiently large $j$, 
 \begin{align*}
  F_{\delta}(z^{k,j})-F_{\delta}(z^k)\le-\frac{\gamma_{k,j}\!-\!1\!-\!L_k-\delta}{2}\|x^{k,j}-x^k\|^2
  +\big[(L_k\!+\!\gamma_{k,j})^2\beta_{k,j}^2-\frac{\delta}{2}\big]\|x^{k}\!-\!x^{k-1}\|^2,
 \end{align*}
 which together with the above \eqref{aim-ineq41} implies that 
 \begin{align*}
 \frac{(1\!-\!\alpha)\gamma_{k,j}\!-\!L_k\!-\!\delta\!-\!1}{2}\|x^{k,j}\!-\!x^k\|^2
 +\big[\frac{(1\!-\!\alpha)\delta}{2}\!-\!(L_k\!+\!\gamma_{k,j})^2\beta_{k,j}^2\big]\|x^k\!-\!x^{k-1}\|^2\!<\! 0.
 \end{align*}
 Recall that $\alpha\in(0,1)$, $\lim_{j\to\infty}\gamma_{k,j}=\infty,\lim_{j\to\infty}\beta_{k,j}=0$ and $\lim_{j\to\infty}\gamma_{k,j}\beta_{k,j}$ $=0$. When $j$ is sufficiently large, it is impossible for the above inequality to hold. The proof is finished.  
\end{proof}

Let $\{z^k\}_{k\in\mathbb{N}}$ be the sequence of Algorithm \ref{PGenls}. By step 2 of Algorithm \ref{PGenls}, following the proof of Lemma \ref{gammak-bound} leads to $\{z^k\}_{k\in\mathbb{N}}\subset\mathcal{L}_{F_{\delta}}(z^0)\!:=\{z\in\!\mathbb{X}\times\mathbb{X}\,|\,F_{\delta}(z)\le F_{\delta}(z^0)\}$. Now if $\mathcal{L}_{F_{\delta}}(z^0)$ is a bounded set, then $\{z^k\}_{k\in\mathbb{N}}$ is bounded, so is $\{\gamma^k\}_{k\in\mathbb{N}}$ by Lemma \ref{PGenls-gamkbound}. 
\begin{lemma}\label{PGenls-gamkbound}
 Suppose that $\mathcal{L}_{F_{\delta}}(z^0)$ is bounded. Then the sequences $\{z^k\}_{k\in\mathbb{N}}$ and $\{\gamma_k\}_{k\in\mathbb{N}}$ are bounded, so there exists $\gamma_{*}\!>\!0$ such that $\gamma_k\!\le\!\gamma_{*}$ for all $k\in\mathbb{N}$.
\end{lemma}
\begin{proof}
 It suffices to prove that $\{\gamma_k\}_{k\in\mathbb{N}}$ is bounded. Suppose on the contrary that $\{\gamma_k\}_{k\in\mathbb{N}}$ is unbounded. Then, there exists $\mathcal{K}\subset\mathbb{N}$ such that $\lim_{\mathcal{K}\ni k\to\infty}\gamma_k=\infty$. For each $k\in\mathcal{K}$, write $\widehat{\gamma}_k\!:=\varrho^{-1}\gamma_k=\gamma_{k,j_k-1}$ and $\widehat{\beta}_k\!:=\nu^{-1}\beta_k=\beta_{k,j_k-1}$. Apparently, $\lim_{\mathcal{K}\ni k\to\infty}\widehat{\gamma}_{k}=\infty$ and $\lim_{\mathcal{K}\ni k\to\infty}\widehat{\beta}_{k}=0=\lim_{\mathcal{K}\ni k\to\infty}\widehat{\gamma}_k\widehat{\beta}_{k}$. From step 2 of Algorithm \ref{PGenls}, for each $k\in\mathcal{K}$,
 \begin{equation}\label{F-ineq41}
 F_{\delta}(z^{k,j_k-1})>F_{\delta}(z^{\ell(k)})-\frac{\alpha\widehat{\gamma}_{k}}{2}\|x^{k,j_{k}-1}-x^k\|^2-\frac{\alpha\delta}{2}\|x^{k}-x^{k-1}\|^2. 
 \end{equation}
 Note that \eqref{optim-xkj} holds for each $k\in\mathcal{K}$ and $j=j_k\!-\!1$. Along with $y^{k,j_k-1}\!=x^k+\widehat{\beta}_k(x^k-x^{k-1})$, 
 \begin{align}\label{optim-xkj1}
  &\quad\langle\nabla\!f(y^{k,j_k-1}),x^{k,j_k-1}-y^{k,j_k-1}\rangle+\frac{\widehat{\gamma}_k}{2}\|x^{k,j_k-1}-y^{k,j_k-1}\|^2+g(x^{k,j_k-1})\nonumber\\
  &\le \langle\nabla\!f(y^{k,j_k-1}),x^{k}-y^{k,j_k-1}\rangle+\frac{\widehat{\gamma}_k\widehat{\beta}_k^2}{2}\|x^{k}-x^{k-1}\|^2+g(x^{k})\quad\forall k\in\mathcal{K}.
 \end{align}
 Since $\{x^k\}_{k\in\mathbb{N}}$ is bounded and $g$ is bounded from below, using $\lim_{\mathcal{K}\ni k\to\infty}\widehat{\gamma}_k=\infty$ and $\lim_{\mathcal{K}\ni k\to\infty}\widehat{\gamma}_k\widehat{\beta}_{k}=0$, we deduce from \eqref{optim-xkj1} that $\lim_{\mathcal{K}\ni k\to\infty}\|y^{k,j_k-1}-x^{k,j_k-1}\|=0$. Note that $\{y^{k,j_k-1}\}_{k\in\mathcal{K}}$ is bounded due to $y^{k,j_k-1}\!=x^k+\widehat{\beta}_k(x^k-x^{k-1})$, so is the sequence $\{x^{k,j_k-1}\}_{k\in\mathcal{K}}$. Together with $\{x^k\}_{k\in\mathcal{K}}\cup\{x^{k,j_k-1}\}_{k\in\mathcal{K}}\subset{\rm dom}\,g\subset\mathcal{O}$, there exists a compact set $\Gamma\subset\mathcal{O}$ such that $y^{k,j_k-1}\in\Gamma$, $x^k\in\Gamma$ and $x^{k,j_k-1}\in\Gamma$ for sufficiently large $k\in\mathcal{K}$. Recall that $\nabla\!f$ is strictly continuous on $\mathcal{O}$, so is Lipschitz continuous on $\Gamma$. Denote by $\widehat{L}_{\!f}$ the Lipschitz constant of $\nabla\!f$ on $\Gamma$. Then, there exist $\widehat{k}\in\mathbb{N}$ such that for all $\mathcal{K}\ni k\ge\widehat{k}$,
 \begin{align*}
 &\|\nabla f(x^k)-\nabla f(y^{k,j_k-1})\|\le \widehat{L}_{\!f}\|x^k-y^{k,j_k-1}\|,\\
 &f(x^{k,j_k-1})-f(x^{k})-\langle\nabla\!f(x^{k}),x^{k,j_k-1}-x^{k}\rangle\le ({\widehat{L}_{\!f}}/{2})\|x^{k,j_k-1}-x^{k}\|^2.
 \end{align*}
 Now following the same argument as in the proof of Lemma \ref{well-defined2} leads to
 \begin{align*}
 \frac{(1\!-\!\alpha)\widehat{\gamma}_k\!-\!\widehat{L}_{\!f}\!-\!\delta\!-\!1}{2}\|x^{k,j_k-1}\!-\!x^k\|^2
 +\big[\frac{(1-\alpha)\delta}{2}-(L_k\!+\!\widehat{\gamma}_k)^2\widehat{\beta}_{k}^2\big]\|x^k-x^{k-1}\|^2< 0.
 \end{align*}
 Recall that $\lim_{\mathcal{K}\ni k\to\infty}\widehat{\gamma}_{k}=\infty$ and $\lim_{\mathcal{K}\ni k\to\infty}\widehat{\beta}_{k}=0=\lim_{\mathcal{K}\ni k\to\infty}\widehat{\gamma}_k\widehat{\beta}_{k}$. When $k\in\mathcal{K}$ is sufficiently large, it is impossible for the above inequality to hold. The conclusion follows.  
\end{proof}

Now we are ready to establish the full convergence of the sequence $\{x^k\}_{k\in\mathbb{N}}$ by proving that $\{x^k\}_{k\in\mathbb{N}}$ and $\{z^k\}_{k\in\mathbb{N}}$ together conform to H1-H4 with $\Phi=\Theta=F_{\delta}$.  
 \begin{theorem}\label{theorem42}
 Suppose that $\mathcal{L}_{F_{\delta}}(z^0)$ is bounded, and that $F_{\delta}$ is a KL function. Then, the sequence $\{z^k\}_{k\in\mathbb{N}}$ of Algorithm \ref{PGenls} is convergent with limit, say, $(x^*,x^*)$, and $x^*$ is a stationary point of \eqref{DC-prob} with $h\equiv 0$. If $F$ is a KL function of exponent $\theta\in[1/2,1)$, there exist $\gamma>0$ and $\varrho\in(0,1)$ such that \eqref{aim-ineq-rate} holds for all $k$ large enough. 
 \end{theorem}
 \begin{proof}
 By \cite[Theorem 3.6]{LiPong18}, if $F$ is a KL function of exponent $\theta\in[1/2,1)$, so is $F_{\delta}$. In view of Theorems \ref{KL-converge} and \ref{KL-rate}, it suffices to prove that $\{x^k\}_{k\in\mathbb{N}}$ and $\{z^k\}_{k\in\mathbb{N}}$ conform to H1-H4 with $\Phi=\Theta=F_{\delta}$. From the iteration steps of Algorithm \ref{PGenls}, $\{z^k\}_{k\in\mathbb{N}}$ satisfies H1 with $a=({\alpha}/{2})\min\{\gamma_{\min},\delta\}$. We complete the proof by the following three steps. 
 
 \noindent
 {\bf Step 1: to prove that $\{z^k\}_{k\in\mathbb{N}}$ satisfies H2 with $\Phi=F_{\delta}$.} Since $\{z^k\}_{k\in\mathbb{N}}$ is bounded, using Lemma \ref{lemma1-Phi} (ii) leads to $\lim_{k\to\infty}\|z^{\ell(k)}-z^{\ell(k)-1}\|=0$. To prove that it satisfies condition H2, pick any convergent subsequence $\{z^{\ell(k_j)}\}_{j\in\mathbb{N}}$ and let $(x^*,x^*)=z^*=\lim_{j\to\infty}z^{\ell(k_j)}$. Clearly, $\lim_{j\to\infty}z^{\ell(k_j)-1}=z^*$ and $\lim_{j\to\infty}x^{\ell(k_j)-1}=x^*=\lim_{j\to\infty}x^{\ell(k_j)-2}$. Note that $\|y^{\ell(k_j)}-x^{\ell(k_j)-1}\|\le\beta_{\rm max}\|x^{\ell(k_j)-1}-x^{\ell(k_j)-2}\|$ for each $j\in\mathbb{N}$. We also have $\lim_{j\to\infty}y^{\ell(k_j)}=x^*$. 
 For each $k\in\mathbb{N}$ and $j\in\mathbb{N}$, from the optimality condition of $x^{\ell(k_j)}$, it follows
 \begin{align}\label{optim2-xkj}
  &\quad\langle\nabla\!f(y^{\ell(k_j)}),x^{\ell(k_j)}-y^{\ell(k_j)}\rangle+\frac{\gamma_{\ell(k_j)-1}}{2}\|x^{\ell(k_j)}-y^{\ell(k_j)}\|^2+g(x^{\ell(k_j)})\nonumber\\
  &\le \langle\nabla\!f(y^{\ell(k_j)}),x^*-y^{\ell(k_j)}\rangle+\frac{\gamma_{\ell(k_j)-1}}{2}\|x^*-y^{\ell(k_j)}\|^2+g(x^*).
 \end{align}
 From $x^*\!=\!\lim_{j\to\infty}y^{\ell(k_j)}\!=\!\lim_{j\to\infty}x^{\ell(k_j)}$ and Lemma \ref{PGenls-gamkbound}, we deduce $\limsup_{j\to\infty}g(x^{\ell(k_j)})\le g(x^*)$. Then,
 $\limsup_{j\to\infty}F_{\delta}(z^{\ell(k_j)})\le\! F_{\delta}(z^*)=F_{\delta}(\lim_{j\to\infty}z^{\ell(k_j)})$, and H2 holds for $\Phi=F_{\delta}$.
 
 \noindent
 {\bf Step 2: to prove that $\{z^k\}_{k\in\mathbb{N}}$ satisfies H3 with $\Theta=F_{\delta}$.} We only need to prove \eqref{dist-Theta} by contradiction. Suppose that there exists an infinite index set $\mathcal{K}\subset\mathbb{N}$ such that 
 \begin{equation}\label{aim-dist1}
 {\rm dist}(0,\partial F_{\delta}(z^k))>M_k\|z^k-z^{k-1}\|\quad{\rm for\ all}\ k\in\mathcal{K}
 \end{equation}
 with $\lim_{\mathcal{K}\ni k\to\infty}M_k=\infty$. 
 For each $k\in\mathbb{N}$, the optimality of $x^k$ implies that for each $k\in\mathbb{N}$,
 \begin{equation}\label{temp-xkyk}
  \begin{pmatrix}
   \nabla f(x^k)-\nabla\!f(y^{k})-\gamma_{k-1}(x^{k}-y^{k})+\delta(x^k-x^{k-1})\\
    \delta(x^{k-1} -x^k)
    \end{pmatrix}\in\partial F_{\delta}(z^k).
  \end{equation}
 Since $\{z^{k}\}_{k\in\mathbb{N}}\subset\mathcal{O}$ is bounded, using Proposition \ref{prop1-Phi} leads to $\lim_{k\to\infty}\|z^{k+1}-z^k\|=0$. Note that $y^{k}=x^{k-1}+\beta_{k-1}(x^{k-1}\!-x^{k-2})$ and $\beta_{k-1}\le\beta_{\rm max}\le1$. There exist $\mathcal{K}_1\subset\mathcal{K}$ and $\widetilde{x}\in{\rm dom}g$ such that $\widetilde{x}=\lim_{\mathcal{K}_1\ni k\to\infty}y^k=\lim_{\mathcal{K}_1\ni k\to\infty}x^{k-1}=\lim_{\mathcal{K}_1\ni k\to\infty}x^{k}$. Note that $\nabla\!f$ is strictly continuous at $\widetilde{x}$. There exists $\widehat{L}_{\!f}>0$ such that $\|\nabla f(x^k)-\nabla f(y^k)\|\le \widehat{L}_{\!f}\|x^k-y^k\|$ for all sufficiently large $k\in\mathcal{K}_1$. Along with \eqref{temp-xkyk} and Lemma \ref{PGenls-gamkbound}, for all sufficiently large $k\in\mathcal{K}_1$, 
 \begin{align*}
 {\rm dist}(0,\partial F_{\delta}(z^k))
 &\le (\widehat{L}_{\!f}\!+\!\gamma_*\!+\!2\delta)\big[\|x^{k}-x^{k-1}\|+\|x^{k-1}\!-\!x^{k-2}\|\big]\\
 &\le \sqrt{2}(\widehat{L}_{\!f}\!+\!\gamma_*\!+\!2\delta)\|z^k\!-\!z^{k-1}\|,
 \end{align*}
 which is a contradiction to \eqref{aim-dist1}. This shows that inequality \eqref{dist-Theta} holds.
 
 \noindent
 {\bf Step 3: to prove that $\{z^k\}_{k\in\mathbb{N}}$ satisfies H4 with $\Phi=F_{\delta}$.} It suffices to prove that $\{F_{\delta}(z^k)\}_{k\in\mathbb{N}}$ satisfies the assumption of Lemma \ref{sequence-cond}. If not, there is an infinite $\mathcal{K}\subset\mathbb{N}$ such that 
 \begin{equation}\label{aim-Fdelta}
  F_{\delta}(z^{k+1})-F_{\delta}(z^{k})>M_k\|z^{k+1}-z^{k}\|^2\quad{\rm for\ all}\ k\in\mathcal{K}
 \end{equation}
 with $\lim_{\mathcal{K}\ni k\to\infty}M_k=\infty$. 
 From the proof in Step 2, there exist an infinite index set $\mathcal{K}_1\subset\mathcal{K}$ and $\widetilde{x}\in{\rm dom}g$ such that $\widetilde{x}=\lim_{\mathcal{K}_1\ni k\to\infty}y^{k+1}=\lim_{\mathcal{K}_1\ni k\to\infty}x^{k}=\lim_{\mathcal{K}_1\ni k\to\infty}x^{k+1}$.
 Since $\nabla\!f$ is Lipschitz continuous at $\widetilde{x}$, there exists $\widehat{L}_{\!f}>0$ such that for all sufficiently large $k\in\mathcal{K}_1$
 \begin{align*}
 &\|\nabla f(x^k)-\nabla f(y^{k+1})\|\le \widehat{L}_{\!f}\|x^k-y^{k+1}\|,\\
 &f(x^{k+1})-f(x^{k})-\langle\nabla\!f(x^{k}),x^{k+1}-x^{k}\rangle\le \frac{\widehat{L}_{\!f}}{2}\|x^{k+1}-x^{k}\|^2.
 \end{align*}
 Following the arguments in the proof of Lemma \ref{well-defined2}, for all sufficiently large $k\in\mathcal{K}_1$, 
 \begin{align*}
 F(x^{k+1})-F(x^k)
  &\le -\frac{\gamma_{k}\!-\!\widehat{L}_{\!f}-1}{2}\|x^{k+1}-x^k\|^2\!+\!\frac{(\gamma_{k}+\widehat{L}_{\!f})^2}{2}\|x^k-x^{k-1}\|^2
 \end{align*}
 Along with the expression of $F_{\delta}$ and Lemma \ref{PGenls-gamkbound}, for all sufficiently large $k\in\mathcal{K}_1$,
 \begin{align*}
 F_{\delta}(z^{k+1})-F_{\delta}(z^k)&\le\frac{\widehat{L}_{\!f}\!+\!\delta\!+\!1}{2}\|x^{k+1}-x^k\|^2\!+\!\frac{(\gamma_{*}+\widehat{L}_{\!f})^2}{2}\|x^k-x^{k-1}\|^2\\
  &\le \frac{\widehat{L}_{\!f}\!+\!\delta\!+\!1+(\gamma_{*}+\widehat{L}_{\!f})^2}{2}\|z^{k+1}-z^k\|^2,
 \end{align*}
 where the last inequality is due to $z^{k+1}=(x^{k+1},x^k)$ for all $k\in\mathbb{N}$. When $k\in\mathcal{K}_1$ is large enough, in view of \eqref{aim-Fdelta}, it is impossible for the above inequality to hold. This proves the claimed result. The desired conclusion follows the above three steps.  
 \end{proof}   

 Recall that Algorithm \ref{PGenls} with $\delta=0$ and $\beta_{\rm max}=0$ becomes the PGnls proposed in \cite{Kanzow22}. Theorem \ref{theorem42} also implies the full convergence of the iterate sequence of PGnls.  
 
 \section{Conclusion}\label{sec5}

It is expected that in order to extend the full convergence results of the
classic GLL-type nonmonotone methods to nonconvex and nonsmooth optimization,
new conditions need to be enforced.
The more restricted such conditions are, the less number of problem types would be covered. 
A recent example is condition (2.6) of \cite{QianPan23}, which is essentially satisfied by weakly convex functions. This is far from what encountered in modern optimization.
Therefore, it is challenging to propose suitable conditions that would cover a wide range of problems.

We believe this paper made a significant progress towards resolving this challenging task. We proposed a principled 
iterative framework H1-H4 applicable to the minimization of a prox-regular KL function $\Phi$, enhancing greatly the existing one \cite{QianPan23} that is known to be satisfied by weakly convex KL functions with a restricted weakly convex parameter. Any sequence $\{x^k\}_{k\in\mathbb{N}}$ and its bounded augmented sequence $\{z^k\}_{k\in\mathbb{N}}$ together complying with H1-H4 is proved to have a full convergence, and when $\Theta$ is a KL function of exponent $\theta\in(0, 1)$, the convergence admits a linear rate if $\theta\in(0, 1/2]$ and a sublinear rate if $\theta\in(1/2, 1)$. We demonstrated that the iterate sequences generated by
 Algorithms \ref{NPG_major} and \ref{PGenls} conform to H1-H4.
 Consequently, we achieved, for the first time, the full convergence of the iterate sequence of NPG$_{\rm major}$ and PGenls. 
 We expect that 
 the proposed novel framework will serve as a guide for developing new algorithms and provide a convenient tool for convergence analysis of existing algorithms for other classes of KL optimization problems.

\backmatter








\section*{Statements and Declarations}


\begin{itemize}
\item {\bf Funding} The second author's work was funded by the Natural Science Foundation of Guangdong Province under project 2023A1515111167. The third author's work was funded by the National Natural Science Foundation of China under project 12371299. 
Qian and Qi were funded by Hong Kong RGC General Research Fund PolyU/15309223 and PolyU AMA Project 230413007.
\item {\bf Competing interests} The authors declare that they have no conflict of interest.
\item {\bf Availability of data and materials} Not applicable.
\end{itemize}

\noindent










\end{document}